\documentclass[11pt,reqno]{amsart}

\usepackage{diagrams}
\usepackage{amsfonts}
\usepackage{indentfirst}
\usepackage{graphicx}
\usepackage{amssymb}
\usepackage{color,xcolor}
\usepackage{graphicx,epstopdf}
\usepackage{epsfig}
\usepackage{subfigure}
\usepackage{caption}
\usepackage{mathrsfs}
\usepackage{bbm}
\usepackage{chngpage}
\usepackage{cite}
\usepackage{multirow}
\usepackage{multicol}
\usepackage{dsfont}
\usepackage{bm}
\usepackage{amssymb,amsmath,amsthm,mathrsfs}%
\usepackage{exscale}
\usepackage{tikz-cd}
\usepackage{relsize}
\usepackage{amssymb}
\usepackage{hyperref}
\usepackage{url}
\usepackage{tikz-cd}
\usepackage{comment}
\usepackage[misc,geometry]{ifsym} 
\allowdisplaybreaks

\hypersetup{hypertex=green,
            colorlinks=true,
            linkcolor=green,
            anchorcolor=green,
            citecolor=red}

\theoremstyle{definition}
\newtheorem{theorem}{Theorem}[section]

\newtheorem{remark}{Remark}[section]
\newtheorem{lemma}{Lemma}[section]

\numberwithin{equation}{section}%
\numberwithin{table}{section}%
\numberwithin{figure}{section}

\def\3bar{{|\hspace{-.02in}|\hspace{-.02in}|}}

\renewcommand\div{\operatorname{div}}

\newcommand\curl{\operatorname{curl}}

\newcommand\ran{\operatorname{ran}}
\newcommand\Span{\operatorname{span}}

\def\d{\text{d}}

\begin{document}
\title[]{Simple curl-curl-conforming finite elements in two dimensions}


\author{Kaibo Hu}
\email{khu@umn.edu}
\address{School of Mathematics,
University of Minnesota, 
Minneapolis, MN 55455,
USA.}
\author{Qian Zhang (\Letter)}
\email{go9563@wayne.edu}
\address{Department of Mathematics, Wayne State University, Detroit, MI 48202, USA. }

\author{Zhimin Zhang}
\email{zmzhang@csrc.ac.cn; zzhang@math.wayne.edu}
\address{Beijing Computational Science Research Center, Beijing, China; Department of Mathematics, Wayne State University, Detroit, MI 48202, USA}
\thanks{This work is supported in part by the National Natural Science Foundation of China grants NSFC 11871092 and NSAF U1930402.}


\date{\today}
\begin{abstract} 
We construct smooth finite element de Rham complexes in two space dimensions. This 
leads to three families of curl-curl conforming finite elements, two of which contain two existing families. The simplest triangular and rectangular finite elements  have  only 6 and 8 degrees of freedom, respectively.  
Numerical experiments for each family demonstrate the convergence and efficiency of the elements for solving the quad-curl problem.

\end{abstract}	
\maketitle
\section{Introduction}
In this paper, we construct and analyze three families of curl-curl conforming ($H(\curl^2)$-conforming) finite elements in two space dimensions (2D) and use these elements to solve the quad-curl problem. 

The quad-curl equation appears in various models, such as the inverse electromagnetic scattering theory \cite{Cakoni2017A,Monk2012Finite, Sun2016A} and magnetohydrodynamics \cite{Zheng2011A}. The corresponding quad-curl eigenvalue problem plays a fundamental role in the analysis and computation of the electromagnetic interior transmission eigenvalues \cite{sun2011iterative}. Some methods have been developed for the source problem and the eigenvalue problem in, e.g., \cite{WZZelement, Zheng2011A,Sun2016A,Qingguo2012A,Brenner2017Hodge,quadcurlWG, Zhang2018M2NA,Chen2018Analysis164, Zhang2018Regular162,SunZ2018Multigrid102,WangC2019Anew101,BrennerSC2019Multigrid100,quad-curl-eig-posterior}. Two of the authors and their collaborator  have recently developed, for the first time, a family of curl-curl conforming finite elements \cite{WZZelement}. 
To reduce the number of degrees of freedom (DOFs), they used incomplete polynomials. 
 The  polynomial degree $k$ starts from 4 for triangular elements and 3 for rectangular elements, respectively, and the lowest-order elements of both shapes have 24 DOFs. Moreover, in \cite{quad-curl-eig-posterior}, they collaborated with J. Sun and constructed another family of curl-curl conforming triangular elements with  complete  polynomials.  The polynomial degree $k$ starts from 4 and hence the lowest-order element has 30 DOFs. In this paper, in addition to the construction of new $H(\curl^2)$-conforming elements, we will also fit the two existing families into complexes and extend them to lower-order cases.

The discrete de Rham complex is now an important tool for the construction of finite elements and analysis of numerical schemes, c.f., \cite{arnold2018finite, arnold2010finite, arnold2006finite, hiptmair1999canonical, neilan2015discrete,christiansen2018nodal}. In this direction, the finite element periodic table \cite{arnold2014periodic} includes various successful finite elements for computational electromagnetism or diffusion problems. Motivated by problems in fluid and solid mechanics, there is an increased interest in constructing finite element de Rham complexes with enhanced smoothness, sometimes referred to as Stokes complexes \cite{falk2013stokes,christiansen2016generalized}. In this paper, for the discretization of the quad-curl problem, we will consider another variant of the de Rham complex, i.e.,
\begin{equation}\label{2D:quad-curl}
\begin{tikzcd}
0 \arrow{r} &\mathbb{R} \arrow{r}{\subset} & H^{1}(\Omega)   \arrow{r}{\nabla} & H(\curl^2; \Omega)  \arrow{r}{\nabla\times} & H^{1}(\Omega) \arrow{r}&0,
 \end{tikzcd}
\end{equation}
where  $\Omega$ is a bounded Lipschitz domain in $\mathbb{R}^{2}$ and 
$$
H(\curl^{2}; \Omega):=\{\bm u \in {\bm L}^2(\Omega):\; \nabla \times \bm u \in L^2(\Omega),\;\bm{\nabla \times}\nabla \times \bm u \in \bm L^2(\Omega)\}.
$$
 For simplicity of presentation, throughout this paper we will assume that $\Omega$ is contractible. Then the exactness of \eqref{2D:quad-curl} follows from standard results in, e.g., \cite{arnold2018finite}.

This complex point of view makes it possible to achieve the goal of this paper, i.e., constructing simple curl-curl conforming elements with fewer degrees of freedom, compared to, e.g., those in \cite{WZZelement} and \cite{quad-curl-eig-posterior}. From this complex perspective, we also fit the quad-curl problem and its finite element approximations in the framework of the finite element exterior calculus (FEEC) \cite{arnold2018finite,arnold2006finite}. Thus a number of tools from FEEC can be used for the numerical analysis. For example, we construct interpolation operators that commute with the differential operators. Then the convergence result follows from a standard argument. 

 Specifically, the new finite elements fit into a subcomplex of \eqref{2D:quad-curl}:
\begin{equation}\label{discrete-complex}
\begin{tikzcd}
0 \arrow{r} &\mathbb{R} \arrow{r}{\subset} & \Sigma_h  \arrow{r}{\nabla} & V_{h}  \arrow{r}{\nabla\times} & W_h \arrow{r}&0.
 \end{tikzcd}
\end{equation}
In \eqref{discrete-complex}, we choose Lagrange finite element spaces for $\Sigma_h$ and Lagrange elements enriched with bubbles for $W_h$. The space $V_h\subset H(\curl^2;\Omega)$ is thus obtained as the gradient of $\Sigma_h$ plus a complementary part, mapped onto $W_h$ by $\curl$. We will use $V_h$ as a conforming finite element for solving the quad-curl problem below. Among the three versions of $V_h$ which we will construct in this paper, the simplest elements have only 6 DOFs for a triangle and 8 DOFs for a rectangle.  To the best of our knowledge, these elements have the smallest number of DOFs among all the existing curl-curl  conforming finite elements. 




The significance of this new development is threefold: 1) It develops new families of curl-curl conforming elements; 2) It relates the curl-curl conforming elements to the FEEC via the de Rham complex, thereby allowing further systematic development of new elements; 3) It reduces element DOF of the existing lowest-order curl-curl conforming element from 24 to 6 and 8 for triangular and rectangular elements, respectively, which makes commercial adoption of the elements feasible.

The remaining part of the paper is organized as follows. In Section 2, we present notations and preliminaries.
In Section 3, we define shape functions and local exact sequences by the Poincar\'e operators and prove their properties. 
In Section 4, we construct a new family of curl-curl conforming finite elements and in Section 5 we extend two existing families to lower-order cases by fitting them into complexes.  In Section 6, we provide numerical examples to verify the correctness and efficiency of our method. Finally, concluding remarks and future work are given in Section 7.

\section{Preliminaries}

Let $\Omega\in\mathbb{R}^2$ be a  contractible Lipschitz domain. We adopt standard notations for Sobolev spaces such as $H^m(D)$ or $H_0^m(D)$ on a 
simply-con-nected sub-domain $D\subset\Omega$ equipped with the norm $\left\|\cdot\right\|_{m,D}$ and the semi-norm $\left|\cdot\right|_{m,D}$. If $m=0$,  the space $H^0 (D)$ coincides with $ L^2(D)$ equipped with the norm $\|\cdot\|_{D}$, and when $D=\Omega$, we drop the subscript $D$. We use  $\bm H^m(D)$   and ${\bm L}^2(D)$ to denote the vector-valued Sobolev spaces $\left[H^m(D)\right]^2$ and $\left[L^2(D)\right]^2$.

Let ${\bm u}=(u_1, u_2)^T$ and ${\bm w}=(w_1, w_2)^T$, where the superscript $T$ denotes the transpose.
Then ${\bm u} \times {\bm w} = u_1 w_2 - u_2 w_1$ and $\nabla \times {\bm u} = \partial_{x_1} u_2  - \partial_{x_2} u_1 $.
For a scalar function $v$, $\bm{\nabla \times} v = (\partial_{x_2} v , - \partial_{x_1} v)^T$.
We denote $(\nabla\times)^2\bm u=\bm\nabla\times\nabla\times\bm u$.

We define
\begin{align*}
&H(\text{curl};D):=\{\bm u \in {\bm L}^2(D):\; \nabla \times \bm u \in L^2(D)\},\\
H(\text{curl}^2;D)&:=\{\bm u \in {\bm L}^2(D):\; \nabla \times \bm u \in L^2(D),\;\bm{\nabla \times}\nabla \times \bm u \in \bm L^2(D)\},
\end{align*}
with the scalar products and norms 
\[(\bm u,\bm v)_{H(\curl^s;D)}=(\bm u,\bm v)+\sum_{j=1}^s((\nabla\times)^j\bm u,(\nabla\times)^j \bm v),\]
and
\[\left\|\bm u\right\|_{H(\curl^s;D)}=\sqrt{(\bm u,\bm u)_{H(\curl^s;D)}},\]
with $s=1,2$.

We use $Q_{i,j}(D)$ to denote the polynomials with two variables $(x_1, x_2)$ where the maximal degree is $i$ in $x_1$ and $j$ in $x_2$. For simplicity, we drop a  subscript $i$ when $i=j$.  We use $P_i(D)$ to represent the space of polynomials on $D$ with degree of no larger than $i$ and $\bm P_i(D)=\left[P_i(D)\right]^2$. We denote $\widetilde P_i(D)$ as the space of homogeneous polynomials.

Let \,$\mathcal{T}_h\,$ be a partition of the domain $\Omega$
consisting of rectangles or triangles. We denote $h_K$ as the diameter of an element $K \in
\mathcal{T}_h$ and $h$ the mesh size of $\mathcal {T}_h$.
We use $C$ to denote a generic positive $h$-independent constant. 

Let $\mathfrak{p}: C^{\infty}(\mathbb{R}^{2})\mapsto \left [C^{\infty}(\mathbb{R}^{2})\right ]^{2}$ be an operator which maps a scalar function to a vector field:
$$
\mathfrak{p} u:=\int_{0}^{1}t \bm{x}^{\perp}u(t\bm x)\, dt,
$$
where
$$
\bm{x}:=(x_{1}, x_{2})^T, \text{ and } \bm{x}^{\perp}:=(-x_{2}, x_{1})^T.
$$
As a special case of the Poincar\'{e} operators (c.f. \cite{hiptmair1999canonical, christiansen2016generalized}), $\mathfrak{p}$ has the following properties:
\begin{itemize}
\item polynomial preserving property: if $u\in {P}_{r}(\mathbb{R}^{2})$, then $\mathfrak{p}u\in \bm {P}_{r+1}(\mathbb{R}^{2})$;
\item the null-homotopy identity
\end{itemize}
\begin{equation}\label{null-homotopy}
\nabla\times \mathfrak{p}u=u,~ \forall u\in C^{\infty}(\mathbb R^2).
\end{equation}

To obtain conforming finite elements, we need a modified version of the Poincar\'e operators $\widetilde{\mathfrak p}$ such that for $u\in P_k(K)$ or $Q_k(K)$ with $k\geq 1$, the tangential component of $\widetilde{\mathfrak p} u$ is constant on each edge of $K$.
To this end, we first prove the following lemma.

\begin{lemma}
	For $u\in P_k(K)$ or $Q_k(K)$ with $k\geq 1$, there exists a function $\varphi_u\in P_{k+1}(K)$ or $Q_{k+1}(K)$ such that $(\mathfrak p u-\nabla \varphi_u)\cdot \bm \tau_{e}\in P_0(e)$ on each edge $e$ of $K.$  
\end{lemma}
\begin{proof}
On each edge $e$, specify  $\psi_e\in P_{k}(e)$ such that $\nabla_{e}\psi_e =\mathfrak p u\cdot \bm \tau_{e}-\frac{1}{|e|}\int_{e}\mathfrak p u\cdot \bm \tau_{e}\d s\in P_{k}(e)$ and $\psi_e=0$ at two ends of $e$. Here $\nabla_{e}$ is the tangential derivative along $e$. By the Lagrange interpolation, we can construct the function $\varphi_u$ such that $\varphi_u=0$ at each vertex and $\varphi_u|_{e}=\psi_e \text{ on }e$. Then  $(\mathfrak p u-\nabla \varphi_u)\cdot \bm \tau_{e}=\frac{1}{|e|}\int_{e}\mathfrak p u\cdot \bm \tau_{e}\d s \in P_0(e)$. 
\end{proof}
The modified Poincar\'{e} operator is then defined as 
\[\widetilde{\mathfrak p}u={\mathfrak p}u-\nabla \varphi_u.\]

We review some basic facts from homological algebra; further details can be found, for instance, in \cite{arnold2006finite}. A differential complex is a sequence of spaces $V^{i}$ and operators $d^{i}$ such that
\begin{equation}\label{general-complex}
\begin{tikzcd}
0 \arrow{r} &V^{1} \arrow{r}{d^{1}} & V^{2}  \arrow{r}{d^{2}} & \cdots \arrow{r}{d^{n-1}} & V^{n} \arrow{r}{d^{n}}&0,
 \end{tikzcd}
\end{equation}
satisfying the complex property $d^{i+1}d^{i}=0$ for $i=1, 2, \cdots, n-1$. Let $\ker(d^{i})$ be the kernel space of the operator $d^{i}$ in $V^{i}$, and $\ran(d^{i})$ be the image of the operator $d^{i}$ in $V^{i+1}$. Due to the complex property, we have $\ker(d^{i})\subset \ran(d^{i-1})$ for each $i\geq 2$. Furthermore, if $\ker(d^{i})= \ran(d^{i-1})$, we say that the complex \eqref{general-complex} is exact at $V^{i}$. At the two ends of the sequence, the complex is exact at $V^{1}$ if $d^{1}$ is injective (with trivial kernel), and is exact at $V^{n}$ if $d^{n-1}$ is surjective (with trivial cokernel).  The complex \eqref{general-complex} is called exact if it is exact at all the spaces $V^{i}$. If each space in \eqref{general-complex} has finite dimensions, then a necessary (but not sufficient) condition for the exactness of \eqref{general-complex} is the following dimension condition:
$$
\sum_{i=1}^{n} (-1)^{i}\dim (V^{i})=0. 
$$


\section{Local spaces and polynomial complexes}

To define a finite element space, we must supply, for each element $K\in\mathcal{T}_h$, the space of shape functions and the DOFs. We will use the following complexes as the local function spaces on each $K\in \mathcal{T}_h$ for \eqref{discrete-complex}:
\begin{equation}\label{local-complex}
\begin{tikzcd}
0 \arrow{r} &\mathbb{R} \arrow{r}{\subset} & \Sigma^{r}_h(K)  \arrow{r}{\nabla} & V^{r-1, k}_{h}(K)   \arrow{r}{\nabla\times} & W_h^{k-1}(K) \arrow{r}&0.
 \end{tikzcd}
\end{equation}
Let $\Sigma^{r}_h(K)$ be $P_{r}(K)$ for a triangle element or $Q_{r}(K)$ for a rectangle element. 
For a triangle element $K$, we set
$$
W_h^{k-1}(K)=
\begin{cases}
	P_{k-1}(K),& k\geq 4,\\
	P_{k-1}(K)\oplus \Span\{B_t\},& k=2,3,
	\end{cases}
$$
where $B_t=\lambda_1\lambda_2\lambda_3$ with the barycentric coordinate $\lambda_i$. For a rectangle element $K$, we set
$$
W_h^{k-1}(K)=
\begin{cases}
	Q_{k-1}(K),& k\geq 3,\\
	Q_{k-1}(K)\oplus \Span\{B_r\},& k=2,
\end{cases}
$$
where $B_r=h_x^{-2}h_y^{-2}\left(x-x_l\right)\left(x-x_r\right)\left(y-y_d\right)\left(y-y_u\right)$ with the element $K=(x_l,x_r)$ $\times(y_d,y_u)$ and $h_x=x_r-x_l$, $h_y=y_u-y_d$. 
%
%
%
%
We define
\begin{equation}\label{Vh2}
	\begin{aligned}
	V_h^{r-1, k}(K)&=\nabla \Sigma^{r}_h(K)\oplus \mathfrak{p}W^{k-1}_h(K),\ \text{when } r=k, k\geq 4 \text{ and }r=k+1,\\ 
	V_h^{r-1, k}(K)&=\nabla \Sigma^{r}_h(K)\oplus \widetilde{\mathfrak{p}}W^{k-1}_h(K),\ \text{when } r=k, k=2,3 \text{ and }r=k-1.
\end{aligned}
\end{equation}

\begin{remark}
  For a rectangular element, we can also use the serendipity elements $\mathcal{S}_r(K)=P_r(K)\oplus \Span\{x_1^rx_2,x_1x_2^r\}$ \cite{arnold2011serendipity} for $\Sigma_h^r(K)$ and use $\mathcal{S}_{k-1}(K)$ when $k\geq 5$ ($\mathcal{S}_{k-1}(K)\oplus \Span\{B_r\}$ when $k<5$) for $W_h^{k-1}$.
  This leads to another three families of rectangular elements with fewer DOFs and the same accuracy.
\end{remark}

\begin{remark}
	The Koszul operator $\kappa u:=u\bm x^{\perp}$ has similar properties as the Poincar\'e  operator \cite{arnold2018finite, arnold2006finite}. Therefore, for polynomial bases in $W_h^{k-1}(K)$ other than the bubbles $B_t$ or $B_r$, we can replace the Poincar\'e operator $\mathfrak p$ by the Koszul operator. For the bubble function $B_t$ on the reference triangle with vertices $(0,0)$, $(0,1)$ and $(1,0)$ or  $B_r$ on the reference rectangle $(-1,1)\times (-1,1),$ we have
\begin{align*}
		\mathfrak p B_t=&\frac{x_1x_2(4x_1 + 4x_2 - 5)}{20}\bm x^{\perp}, \ \
		\mathfrak p B_r=\frac{2x_1^2x_2^2 - 3x_1^2 - 3x_2^2 + 6}{12}\bm x^{\perp}.\\
		\widetilde{\mathfrak p} B_t=& \mathfrak p B_t-\nabla (x_1x_2(2x_1-1)),\ \
		\widetilde{\mathfrak p} B_r=\mathfrak p B_r-\nabla (x_2x_1^3-x_1x_2^3)/36.
	\end{align*}
	\end{remark}
By the null-homotopy identity \eqref{null-homotopy}, the right hand side of \eqref{Vh2} is a direct sum. 
\begin{lemma}
The local sequence \eqref{local-complex} is a complex and exact. 
\end{lemma}

\begin{proof}
Because of the definition \eqref{Vh2} of $V_h^{r-1,k}(K)$  and the null-homotopy identity \eqref{null-homotopy},  we have $\nabla \Sigma^{r}_h(K)\subseteq V^{r-1, k}_h(K)$ and $\nabla\times V^{r-1, k}_h(K)= W_h^{k-1}(K)$. This shows that  \eqref{local-complex} is a complex. It remains to show the exactness. We first show that, for any $\bm v_h\in V^{r-1, k}_h(K)$ for which $\nabla\times\bm v_h=0$, there exists a $p_h\in \Sigma^{r}_h(K)$ s.t. $\bm v_h=\nabla p_h.$
Since $\bm v_h\in V^{r-1, k}_h(K)$, we have $\bm v_h=\nabla p_h+\mathfrak p w_h$ or $\bm v_h=\nabla p_h+\widetilde{\mathfrak p} w_h$ with $p_h\in \Sigma^{r}_h(K)$ and $w_h\in W^{k-1}_h(K)$. By the null-homotopy identity \eqref{null-homotopy} again, $0=\nabla\times\bm v_h=w_h$ (thus $\varphi_{w_h}=0$ when $\widetilde{\mathfrak p}$ is involved). Therefore, $\bm v_h=\nabla p_h.$ Moreover, the curl operator $\nabla\times: V^{r-1, k}_h(K) \to W^{k-1}_h(K)$ is surjective since $\nabla\times V^{r-1, k}_h(K)=W^{k-1}_h(K)$.
 \end{proof}

In the following lemma, we show that $V_h^{r-1,k}(K)$ contains some polynomial subspaces.
\begin{lemma}\label{Vh} Suppose that $r\leq k+1$. Then $\bm P_{r-1}(K)\subseteq V^{r-1, k}_h(K)$.
\end{lemma}
\begin{proof}
	We claim that 
		\begin{align}\label{dcmp_Pr}
		\bm P_{r-1}(K)=\nabla P_{r}(K)\oplus \mathfrak{p} P_{r-2}(K).
	\end{align}
	In fact, $\nabla P_{r}(K)\oplus\mathfrak{p} P_{r-2}(K)\subseteq \bm P_{r-1}(K)$. To show \eqref{dcmp_Pr}, we only need to show 
	\[\dim \nabla P_{r}(K)\oplus\mathfrak{p} P_{r-2}(K)=\dim  \bm P_{r-1}(K).\]
	By the null-homotopy identity \eqref{null-homotopy}, the right hand side is a direct sum. Therefore, 
	\begin{align}
		\dim \nabla P_{r}(K)&\oplus\mathfrak{p} P_{r-2}(K)=\dim  \nabla P_{r}(K)+\dim P_{r-2}(K)\\
		=&\frac{(r+1)(r+2)}{2}+\frac{(r-1)r}{2}-1=r(r+1),
	\end{align}
which is exactly the dimension of $\bm P_{r-1}(K)$. 
	
Combining \eqref{dcmp_Pr} and the fact that $P_{r-2}(K)\subseteq W^{k-1}_h(K)$ and $P_{r}(K)\subseteq \Sigma^{r}_h(K)$, we get $\bm P_{r-1}(K)\subseteq \nabla \Sigma^r_h(K)\oplus \mathfrak{p}W^{k-1}_h(K)=V^{r-1, k}_h(K)$.  Similarly, we can prove the lemma for the case when $\widetilde{\mathfrak{p}}$ is involved.
\end{proof}

In the following sections, we will take different values of $r$ to get various families of curl-curl conforming finite elements and complexes. In Section \ref{sec:new}, we take $r=k-1$, and this leads to a new family of simple elements. In Section \ref{sec:existing}, we introduce the other two families of elements by taking $r=k,k+1$, respectively.

\section{A new family of curl-curl conforming elements  $r=k-1$}\label{sec:new}

In this section, we construct a new family of curl-curl conforming elements $V_{h}^{k-2, k}$ by specifying $r=k-1$ in \eqref{local-complex}, i.e., 
\begin{equation}\label{discrete-complex-k-1}
\begin{tikzcd}
0 \arrow{r} &\mathbb{R} \arrow{r}{\subset} & \Sigma_h^{k-1}  \arrow{r}{\nabla} & V_{h}^{k-2, k}   \arrow{r}{\nabla\times} & W_h^{k-1} \arrow{r}&0.
 \end{tikzcd}
\end{equation}
 For simplicity of presentation, we focus on the triangular elements and only mention the rectangular elements in Remark \ref{rmk:rectangular} below.

\subsection{Degrees of freedom and global finite element spaces}\label{sec:dofs}
We define DOFs for the spaces in \eqref{discrete-complex-k-1}.

The DOFs for the Lagrange element $\Sigma^{r}_{h}$ can be given as follows.
\begin{itemize}
\item Vertex DOFs $ M_{v}({u})$ at all the vertices  $v_{i}$ of $K$:
$$
M_{v}(u)=\left\{u\left({v}_{i}\right) \text{ for all vertices $v_i$}\right\}.
$$
\item Edge DOFs $M_{e}(u)$ on all the edges $e_{i}$ of $K$:
\begin{align*} M_{e}(u)=\left\{\int_{e_i} u v \mathrm{d} s\text { for all } v \in P_{r-2}(e_i) \text { and for all edges }e_i\right\}.
\end{align*}
 
\item Interior DOFs $M_{K}(u)$: 
$$M_{K}(u)=\left\{ \int_K u v \mathrm{d} A \text{ for all } v \in P_{r-3}(K) \text{ or } Q_{r-2}(K) \right\}.$$
\end{itemize}
For $ u \in H^{1+\delta}(\Omega)$ with $\delta >0$,  we can
define an $H^1$ interpolation operator $\pi_h$ by the above DOFs. The restriction of $\pi_h$ on $K$ is denoted as $\pi_K$ and defined by 
 \begin{eqnarray}\label{def-inte-H1}
 M_v( u-\pi_K u)=\{0\},\  M_e(u-\pi_Ku)=\{0\},\ \text{and}\  M_K( u-\pi_K u)=\{0\}.
\end{eqnarray}

The DOFs for $W_{h}^{k-1}$ can be given similarly, with only one additional interior integration DOF on $K$ to take care of the interior bubble. We denote $\tilde\pi_h$ as the $H^1$ interpolation operator to $W_h^{k-1}$ by these DOFs.

For the shape function space $V_h^{k-2, k}(K):=\nabla P_{k-1}(K)\oplus \mathfrak p W^{k-1}_h(K)$ of the triangular elements, we define the following DOFs:
\begin{itemize}
		\item Vertex DOFs $\bm M_{ {v}}({\bm u})$ at all the vertices $ {v}_{i}$ of $K$:
	\begin{equation}\label{tridef1-1}
	\bm M_{ {v}}({\bm u})=\left\{( \nabla\times {\bm u})(  v_{i}),\; i=1,\;2\;,3\right\}.
	\end{equation}
	\item Edge DOFs $\bm M_{ {e}}(  {\bm u})$ at all the edges $ {e}_i$ of $ {K}$ (with the unit  tangential vector $ {\bm \tau}_i$):
	\begin{align}
		 \bm M_{ {e}}(  {\bm u})=&\left\{\int_{e_i} {\bm u}\cdot  {\bm \tau}_i  {q}\mathrm d {s},\ \forall  {q}\in P_{k-2}( {e}_i), i=1,2,3\right\}\nonumber\\
		 \cup&	\left\{\int_{e_i}\nabla\times{\bm u}q\d s,\ \forall  {q}\in P_{k-3}( {e}_i), i=1,2,3\right\}.\label{tridef1-2}
	\end{align}	
    \item Interior DOFs $\bm M_{ {K}}(  {\bm u})$:
	\begin{align}\label{tridef1-3}
	&\bm M_{ {K}}(  {\bm u})=\left\{\int_{ {K}}  {\bm u}\cdot  {\bm q}\mathrm \d A,\ \forall{{\bm q}}\in  \mathcal{D} \right\},
	\end{align}
	where $\mathcal{D}=\bm P_{k-5}( K)\oplus\widetilde{P}_{k-5} {{\bm x}}\oplus\widetilde{P}_{k-4} {{\bm x}}$ with $ {{\bm x}}=( x_1,\;   x_2)^T$ when $k\geq 5$; $\mathcal{D}={P}_{0} {{\bm x}}$ when $k=4$; $\mathcal{D}=\emptyset$ when $k=2,3$.
		\end{itemize}

\begin{center} 
\begin{figure}
\setlength{\unitlength}{1.2cm}
\begin{picture}(5,6.3)(1.8,-4)
\put(0,0){
\begin{picture}(2,2)
\put(-1, 0){\line(1,2){1}} 
\put(0, 2){\line(1,-2){1}}
\put(-1,0){\line(1,0){2}}
\put(-1.,0){\circle*{0.1}}
\put(1.,0){\circle*{0.1}}
\put(0,2){\circle*{0.1}}
\end{picture}
}

\put(1.5, 1){\vector(1, 0){1}}
\put(1.68, 1.15){$\nabla$}

\put(4,0){
\put(-1, 0){\line(1,2){1}} 
\put(0, 2){\line(1,-2){1}}
\put(-1,0){\line(1,0){2}}
\put(-0.16, 2.06){$\curl$}
\put(-1.14, -0.18){$\curl$}
\put(0.88, -0.18){$\curl$}
\put(-0.75, 0.65){\vector(1, 2){0.3}}
\put(0.75, 0.65){\vector(-1, 2){0.3}}
\put(-0.2, -0.05){\vector(1, 0){0.6}}
}

\put(5.5, 1){\vector(1, 0){1}}
\put(5.75, 1.1){{$\nabla\times$}}
\put(8,0){
\begin{picture}(2,2)
\put(-1, 0){\line(1,2){1}} 
\put(0, 2){\line(1,-2){1}}
\put(-1,0){\line(1,0){2}}
\put(-1, 0){\circle*{0.1}}
\put(1, 0){\circle*{0.1}}
\put(-0, 2){\circle*{0.1}}
\put(-0, 0.5){\circle*{0.1}}
\end{picture}
}

\put(0,-3){
\begin{picture}(2,2)
\put(-1, 0){\line(0,2){2}} 
\put(-1, 2){\line(1,0){2}}
\put(1,2){\line(0,-2){2}}
\put(-1,0){\line(2,0){2}}
\put(-1.,2){\circle*{0.1}}
\put(-1.,0){\circle*{0.1}}
\put(1,2){\circle*{0.1}}
\put(1,0){\circle*{0.1}}
\end{picture}
}

\put(1.5, -2){\vector(1, 0){1}}
\put(1.68, -1.8){$\nabla$}

\put(4,-3){
\put(-1, 0){\line(0,2){2}} 
\put(-1, 2){\line(1,0){2}}
\put(1,2){\line(0,-2){2}}
\put(-1,0){\line(2,0){2}}
\put(-1.2,2.06){$\curl$}
\put(-1.2,-0.2){$\curl$}
\put(0.8,2.06){$\curl$}
\put(0.8,-0.2){$\curl$}
\put(-0.2, 2.06){\vector(1, 0){0.6}}
\put(1.06, 0.65){\vector(0, 1){0.6}}
\put(-1.06, 0.65){\vector(0, 1){0.6}}
\put(-0.2, -0.05){\vector(1, 0){0.6}}
}

\put(5.5, -2){\vector(1, 0){1}}
\put(5.75, -1.8){{$\nabla\times$}}
\put(8,-3){
\begin{picture}(2,2)
\put(-1, 0){\line(0,2){2}} 
\put(-1, 2){\line(1,0){2}}
\put(1,2){\line(0,-2){2}}
\put(-1,0){\line(2,0){2}}
\put(-1.,2){\circle*{0.1}}
\put(-1.,0){\circle*{0.1}}
\put(1,2){\circle*{0.1}}
\put(1,0){\circle*{0.1}}
\put(-0, 1){\circle*{0.1}}
\end{picture}
}
\put(0,-4){$\Sigma_h^1$}
\put(1.5, -4){\vector(1, 0){1}}
\put(1.68, -3.8){$\nabla$}
\put(4,-4){$V_h^{0,2}$}
\put(5.5, -4){\vector(1, 0){1}}
\put(5.75, -3.8){{$\nabla\times$}}
\put(8,-4){$W_h^1$}
\end{picture}

\caption{The lowest-order ($k=2$) finite element complex \eqref{discrete-complex-k-1} in 2D. }
\label{fig:third-family}
\end{figure}
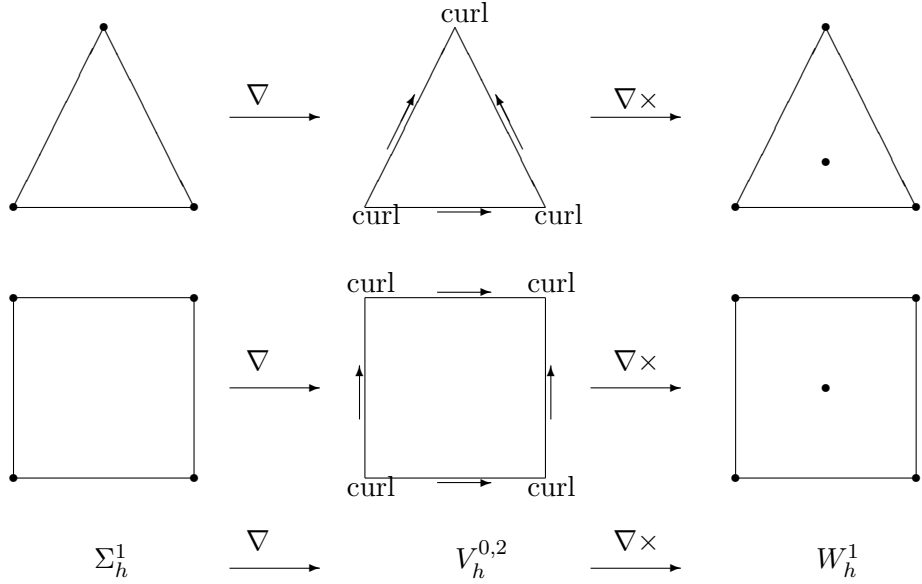
\end{center}

\begin{lemma}\label{well-defined-conditions}
	The DOFs \eqref{tridef1-1}-\eqref{tridef1-3} are well-defined for any ${\bm u}\in \bm H^{1/2+\delta}({K})$ and ${\nabla}\times{\bm u}\in H^{1+\delta}({K})$ with $\delta>0$.
\end{lemma}
The proof of this lemma is the same as that of Lemma 3.4 in \cite{WZZelement}. We omit it here.

\begin{lemma}
The DOFs for $V^{k-2, k}_{h}(K)$ are unisolvent. 
\end{lemma}
\begin{proof}
The decomposition \eqref{Vh2} is a direct sum. Therefore $\dim V^{k-2, k}_{h}(K) =\dim \nabla \Sigma^{k-1}_h(K)+\dim W^{k-1}_h(K)={k(k+1)-1}$ when $k\geq 4$ and $\dim V^{k-2, k}_{h}(K)=k(k+1)$ when $k=2,3$.  By counting the number of DOFs, the DOFs have the same dimension. Then it suffices to show that if all the DOFs vanish on a function $\bm{u}$, then $\bm{u}=0$. To see this, we first observe that $\nabla\times \bm{u}=0$ by the unisolvence of the DOFs of $W^{k-1}_h(K)$. Then $\bm{u}=\nabla\phi\in P_{k-2}(K)$ for some $\phi\in \Sigma^{k-1}_h(K)$.  By the edge DOFs of $V^{k-2, k}_{h}(K)$, 
$\bm u\cdot \bm {\tau}=0$ on edges. Then there exists some $\psi \in P_{k-4}(K)$ such that $\phi=\lambda_{1}\lambda_{2}\lambda_{3}\psi$. Choosing $\bm{q}\in P_{k-4}(K)\bm x$ for which $\nabla\cdot\bm q=\psi$, we have by the interior DOFs:
\[0=\left(\bm u,\bm q\right)=\left(\nabla \phi,\bm q\right)=-\left(\phi,\nabla\cdot\bm q\right)=\left(\lambda_{1}\lambda_{2}\lambda_{3}{\psi},{\psi}\right).\]
This implies that $\psi=0$ and hence $\phi=0$ and $\bm u=0$.
\end{proof}

Provided $\bm u \in \bm H^{1/2+\delta}(\Omega)$, and $ \nabla \times \bm u \in H^{1+\delta}(\Omega)$ with $\delta >0$ (see Lemma \ref{well-defined-conditions}),  we can
define an $H(\curl^2)$ interpolation operator $\Pi_h$ whose restriction on $K$ is denoted as $\Pi_K$ and defined by
 \begin{eqnarray}\label{def-inte-01}
\bm M_v(\bm u-\Pi_K\bm u)=\{0\},\ \bm M_e(\bm u-\Pi_K\bm u)=\{0\},\ \text{and}\ \bm M_K(\bm u-\Pi_K\bm u)=\{0\},
\end{eqnarray}
where $\bm M_v,\ \bm M_e$ and $\bm M_K$ are the sets of DOFs in \eqref{tridef1-1}-\eqref{tridef1-3}.

Gluing the local spaces by the above DOFs, we obtain the global finite element spaces $\Sigma_h^{k-1}$, $V_h^{k-2, k}$ and $W_h^{k-1}$. 

\begin{lemma} The conformity holds:
$$
V^{k-2, k}_{h}\subset H(\curl^2; \Omega).
$$
\end{lemma}
\begin{proof}
To prove the conformity, we shall show that the tangential component of $\bm u\in V_h^{k-2,k}$ is single-valued on each edge $e$ and $\nabla\times\bm u\in H^1(\Omega).$   
It is straightforward to see $\nabla\times\bm u\in H^1(\Omega)$ since  $\nabla\times V^{k-2, k}_{h}\subseteq W^{k-1}_h\subset H^{1}(\Omega)$.
From the definition of $V_h^{k-2,k}$, the tangential component of $\bm u$ is a polynomial of order $k-2$ on each edge $e$. The $k-1$ DOFs in the first set of \eqref{tridef1-2} can determine uniquely the tangential component of $\bm u$. 
\end{proof}

\subsection{Global finite element complexes for the quad-curl problem}

The global finite element spaces lead to a complex which is exact on contractible domains.
\begin{theorem}
The complex \eqref{discrete-complex-k-1}  is  exact on contractible domains.
\end{theorem}
\begin{proof}
We first show the exactness at $V^{k-2, k}_h$. To this end, we show that, for any $\bm v_h\in V^{k-2, k}_h\subset H(\curl^2;\Omega)$ satisfying $\nabla\times\bm v_h=0$, there exists  $p\in \Sigma^{k-1}_h$ such that $\bm v_h=\nabla p$. Actually, this follows from the exactness of the standard finite element differential forms (e.g., \cite{arnold2018finite}) and the fact that the curl-free part of $V^{k-2, k}_h$ is a subspace of the second N\'ed\'elec space of degree $k-2$.  To prove the exactness at $W^{k-1}_h$, that is to prove the operator $\nabla\times$ from $V^{k-2, k}_h$ to $W^{k-1}_h$ is surjective,  we  count the dimensions. The dimension count of the Lagrange elements reads:
 \[\dim \Sigma^{k-1}_h=\mathcal V+(k-2)\mathcal E+\frac{1}{2}(k-3)(k-2)\mathcal F,\]
where $\mathcal V$, $\mathcal E$, and $\mathcal F$ denote the number of vertices, edges, and 2D cells, respectively.
Moreover, $\dim W^{k-1}_h= \dim \Sigma^{k-1}_h$ for $k \geq 4$ and $\dim W^{k-1}_h= \dim \Sigma^{k-1}_h+\mathcal F$ for $k=2,3$.
From the DOFs \eqref{tridef1-1} -\eqref{tridef1-3}, 
\begin{align*}
	\dim V^{k-2, k}_h&=\mathcal V+(2k-3)\mathcal E+(k^2-5k+5)\mathcal F \text{ for } k\geq 4,\\
&\dim V^{k-2, k}_h=\mathcal V+(2k-3)\mathcal E\text{ for } k =2,3.
\end{align*}
From the above dimension count, we have 
\[\dim V^{k-2, k}_h=\dim W^{k-1}_h+\dim \Sigma^{k-1}_h-1,\]
where we have used Euler's formula $\mathcal V-\mathcal E+\mathcal F=1$.
This completes the proof.
\end{proof}

We summarize the interpolations defined in Section \ref{sec:dofs} in the following diagram.
\begin{equation}\label{2complex}
\begin{tikzcd}
0 \arrow[r] 
& \mathbb{R} \arrow[r,"\subset"]  & H^1(\Omega) \arrow[d ]\arrow[r,"\nabla"]  & H(\curl^2;\Omega)\arrow[r,"\nabla\times"]\arrow[d ]  & H^1(\Omega)\arrow[r]\arrow[d ]& 0\\
 0 \arrow[r] 
& \mathbb{R} \arrow[r,"\subset"]  & W \arrow[d, "\pi_h" ]\arrow[r,"\nabla"]  & V\arrow[r,"\nabla\times"]\arrow[d, "\Pi_h" ]  &W\arrow[r]\arrow[d, "\tilde{\pi}_h" ]& 0\\
0 \arrow[r] & \mathbb{R} \arrow[r,"\subset"]  & \Sigma^{k-1}_h\arrow[r,"\nabla"]  & V^{k-2, k}_h\arrow[r,"\nabla\times"]  & W^{k-1}_h\arrow[r]& 0,
\end{tikzcd}
\end{equation}
where $W$ and $V$ are two subspaces of $H^1(\Omega)$ and $H(\curl^2;\Omega)$ in which $\pi_h$ (or $\tilde\pi_h$) and $\Pi_h$ are well-defined. 

Now we show that the interpolations in \eqref{2complex} commute with the differential operators.  This result will play a key role in the error analysis below for discretizing the quad-curl problem.
\begin{lemma}\label{commute}
The last two rows of the complex \eqref{2complex} are a commuting diagram, i.e.,
\begin{align}
	\nabla\pi_h u&=\Pi_h\nabla u \text{ for all } u\in W,\label{Pih_and_pih}\\
	\nabla\times\Pi_h \bm u&=\tilde\pi_h\nabla\times\bm u \text{ for all } \bm u\in V.\label{Pih_and_tildepih}
\end{align}
\end{lemma}
\begin{proof}
	We only prove \eqref{Pih_and_pih}. A similar trick can be used to prove \eqref{Pih_and_tildepih}.
	From the diagram \eqref{2complex}, we know both $\Pi_h\nabla u$ and $\nabla\pi_h u$ are in the space $V_h^{k-2,k}$. It suffices to prove that the DOFs \eqref{tridef1-1}-\eqref{tridef1-3} for $\Pi_h\nabla u$ and $\nabla\pi_h u$ agree element by element. 
	For a given element $K$ with a vertex $v_i$, we first have
	\[\nabla\times\big(\Pi_h\nabla u-\nabla \pi_h u\big)(v_i)=\nabla\times\big(\nabla u-\nabla \pi_h u\big)(v_i)=0.\]	
	On an edge $e_i$ with a tangent vector $\bm \tau_i$ and two vertices $v_1$ and $v_2$, for any $q\in P_{k-2}(e_i)$,  we derive
	\begin{align*}
		&\int_{e_i}\big(\Pi_h\nabla u-\nabla\pi_h u\big)\cdot\bm \tau_i q\d s
		=\int_{e_i}\big(\nabla u-\nabla\pi_h u\big)\cdot\bm \tau_i q\d s\\
		=p(v_2)&(u-\pi_h u)(v_2)-p(v_1)(u-\pi_h u)(v_1)-\int_{e_i}\big( u-\pi_h u\big)\frac{\partial q}{\partial \bm \tau_i} \d s
		= 0.
	\end{align*}
Here we used integration by parts and the definition of the interpolations.
By the definition of $\Pi_h$, we have 
	\begin{align*}
		\int_{e_i}\nabla\times\big(\Pi_h\nabla u-\nabla\pi_h u\big) q\d s
		= 0.
	\end{align*}
	For the interior DOFs, we see that for any $\bm q\in \bm P_{k-3}(K)$,
	\begin{align*}
		&\int_K \big(\Pi_h\nabla u-\nabla\pi_h u\big)\cdot\bm q\d S=\int_K \big(\nabla u-\nabla\pi_h u\big)\cdot\bm q\d S\\
		=&-\int_K \big( u-\pi_h u\big)\nabla\cdot\bm q\d S+\int_{\partial K} \big( u-\pi_h u\big)\bm q\cdot \bm n\d s=0.
	\end{align*}	 
	This completes the proof.	
	\end{proof}

\begin{theorem}
If $\bm u\in \bm H^{s-1}(\Omega)$ and $\nabla\times\bm u\in  H^s(\Omega)$, $ 1+\delta\leq s \leq k$ with $\delta>0$, then we have the following error estimates for the interpolation $\Pi_h$,
	\begin{align}
	&\left\|\bm u-\Pi_h\bm u\right\|\leq Ch^{s-1}(\left\|\bm u\right\|_{s-1}+\left\|\nabla\times\bm u\right\|_{s}),\label{inter-u}\\
	&\left\|\nabla\times(\bm u-\Pi_h\bm u)\right\|\leq Ch^s\left\|\nabla\times\bm u\right\|_{s},\label{inter-curlu}\\
	&	\left\|(\nabla\times)^2(\bm u-\Pi_h\bm u)\right\|\leq Ch^{s-1}\left\|\nabla\times\bm u\right\|_{s}.
	\end{align}
\end{theorem}
\begin{proof}
From Lemma \ref{Vh}, $\bm P_{k-2}(K)\subseteq V^{k-2, k}_h(K)$ and  $\bm P_{k-1}(K)\subseteq W^{k-1}_h(K)$. By a similar proof of Theorem 3.11 in \cite{WZZelement} and using Lemma \eqref{commute}, we complete the proof. 
\end{proof}
\begin{remark}
	Here, we only provide the approximation property for the interpolation $\Pi_h\bm u$. Since $V_h^{k-2,k}$ is a conforming finite element space, the approximation property of the numerical solution $\bm u_h$ follows immediately from C\'ea's lemma. It is the same for the other two families.
\end{remark}

\begin{remark}\label{rmk:rectangular}
Similarly, we can get a family of rectangular elements.
The DOFs for $\bm{u}\in V^{k-2, k}_{h}(K)=\nabla Q_{k-1}(K)+\widetilde{\mathfrak p}W^{k-1}_h(K)$ are given by the following.
\begin{itemize}
		\item Vertex DOFs $\bm M_{ {v}}({\bm u})$ at all the vertices $ {v}_{i}$ of $K$:
	\begin{equation*}
	\bm M_{ {v}}({\bm u})=\left\{( \nabla\times {\bm u})(  v_{i}),\; i=1,\;2,\cdots,4\right\}.
	\end{equation*}
	\item Edge DOFs $\bm M_{ {e}}(  {\bm u})$ at all the edges $ {e}_i$ of $ {K}$ (with the unit  tangential vector $ {\bm \tau}_i$):
	\begin{align*}
		 \bm M_{ {e}}(  {\bm u})=&\left\{\int_{e_i} {\bm u}\cdot  {\bm \tau}_i  {q}\mathrm d {s},\ \forall  {q}\in P_{k-2}( {e}_i), i=1,2,\cdots,4\right\}\\
		 \cup&	\left\{\int_{e_i}\nabla\times{\bm u}q\d s,\ \forall  {q}\in P_{k-3}( {e}_i), i=1,2,\cdots,4\right\}.
	\end{align*}	
    \item Interior DOFs $\bm M_{ {K}}(  {\bm u})$:
	\begin{align*}
	&\bm M_{ {K}}(  {\bm u})=\left\{\int_{ {K}}  {\bm u}\cdot  {\bm q}\mathrm \d A,\ \forall {\bm q} \in  \mathcal{G}_1\oplus \mathcal{G}_2 \right\},
	\end{align*}
	where $\mathcal{G}_1=\big\{ {\bm q}\ |\  {\bm q}=  \psi {{\bm x}},\ \forall   \psi\in Q_{k-3}(K)\big\}\ \text{and}\ \mathcal{G}_2=\big\{ {\bm q}\ |\   {\bm q}= \bm{\nabla}\times {\varphi},\ \forall  {\varphi}\in  {Q}_{k-3}(K)\slash{\mathbb{R}}\big\}$
	 when $k\geq 3$; $\mathcal{G}_1=\mathcal{G}_2=\emptyset$ when $k=2$.
\end{itemize}
The same theoretical results as the triangular elements can be obtained by a similar argument. 
\end{remark}

\section{Two families of  curl-curl conforming elements with $r=k$ and $r=k+1$}\label{sec:existing}
The curl-curl conforming elements introduced in \cite{WZZelement,quad-curl-eig-posterior} are restricted to high-order cases, i.e., $k\geq 4$ for triangular elements and  $k\geq 3$ for rectangular elements in \cite{WZZelement}, and $k\geq 4$ for the triangular elements in \cite{quad-curl-eig-posterior}. Rectangular elements are missing in \cite{quad-curl-eig-posterior}.

In this section,  we will construct two families of curl-curl conforming elements by setting $r=k$ and $r=k+1$ with $k\geq 2$. The two families of elements contain the elements in \cite{WZZelement,quad-curl-eig-posterior}. Similar properties as in \cite{WZZelement,quad-curl-eig-posterior} hold for the generalizations below. For brevity, we only present the definitions and the approximation properties of the $V_h$ spaces.

\subsection{A family of the curl-curl conforming elements with $r=k$}
By taking $r=k$, we obtain another family of finite element complexes, i.e.,
\begin{equation}\label{discrete-complex-k}
\begin{tikzcd}
0 \arrow{r} &\mathbb{R} \arrow{r}{\subset} & \Sigma_h^{k}  \arrow{r}{\nabla} & V_{h}^{k-1, k}  \arrow{r}{\nabla\times} & W_h^{k-1} \arrow{r}&0.
 \end{tikzcd}
\end{equation}
Recall that $\Sigma_h^{k}$ is the Lagrange finite element space of order $k$, and $V^{k-1, k}_h(K)=\nabla \Sigma_h^{k}(K)\oplus \mathfrak p W_h^{k-1}(K)$ with $k\geq 4$ or $V^{k-1, k}_h(K)=\nabla \Sigma_h^{k}(K)\oplus \widetilde{\mathfrak p} W_h^{k-1}(K)$ with $k=2,3$.
By Lemma \ref{Vh}, $V^{k-1, k}_h(K)$ contains $\bm P_{k-1}(K)$. More precisely, \begin{align*}
 	V^{k-1, k}_h(K)=&\mathcal{R}_k\triangleq\bm P_{k-1}\oplus \big\{\bm u\in\widetilde {\bm P}_k\big|\ \bm u\cdot \bm x =0\big\} \text{ when } k\geq 4 \text{ and } K \text{ is a triangle},\\
 	&V^{k-1, k}_h(K)=Q_{k-1,k}\times Q_{k,k-1} \text{ when } k\geq 3 \text{ and } K \text{ is a rectangle},
 \end{align*}
 which can be proved by a similar argument as for Lemma \ref{Vh}. 

For the triangular elements with $k\geq 4$ (rectangular elements with $k\geq 3$),  $V^{k-1, k}_h$ coincides with the curl-curl conforming elements in \cite{WZZelement}.  Here we extend these finite elements to lower-order by allowing $k=2$ or $3$.   The sequence of the lowest-order case is shown in Fig. \ref{fig:firstfamily}. These elements have 9 DOFs on a triangle and 13 DOFs on a rectangle.
  

\subsubsection{Triangular elements} We define the following DOFs for $V^{k-1, k}_h(K)=\nabla P_k(K)\oplus \mathfrak p W^{k-1}_h(K)$.

\begin{itemize}
		\item Vertex DOFs $\bm M_{ {v}}({\bm u})$ at all the vertices $ {v}_{i}$ of $K$:
	\begin{equation*}
	\bm M_{ {v}}({\bm u})=\left\{( \nabla\times {\bm u})(  v_{i}),\; i=1,\;2,3\right\}.
	\end{equation*}
	\item Edge DOFs $\bm M_{ {e}}(  {\bm u})$ at all the edges $ {e}_i$ of $ {K}$ (with the unit  tangential vector $ {\bm \tau}_i$):
	\begin{align*}
		 \bm M_{ {e}}(  {\bm u})=&\left\{\int_{e_i} {\bm u}\cdot  {\bm \tau}_i  {q}\mathrm d {s},\ \forall  {q}\in P_{k-1}( {e}_i), i=1,2,3\right\}\\
		 \cup&	\left\{\int_{e_i}\nabla\times{\bm u}q\d s,\ \forall  {q}\in P_{k-3}( {e}_i), i=1,2,3\right\}.
	\end{align*}	
    \item Interior DOFs $\bm M_{ {K}}(  {\bm u})$:
	\begin{align*}
	&\bm M_{ {K}}(  {\bm u})=\left\{\int_{ {K}}  {\bm u}\cdot  {\bm q}\mathrm \d A,\ \forall \bm q \in  \mathcal{D} \right\},
	\end{align*}
	where $\mathcal{D}=\bm P_{k-5}( K)\oplus\widetilde{P}_{k-5} {\bm x}\oplus\widetilde{P}_{k-4} {\bm x}\oplus
	\widetilde{P}_{k-3} {{\bm x}}$ when $k\geq 5$; $\mathcal{D}={P}_{k-3} {\bm x}$ when $k=3,4$; $\mathcal{D}=\emptyset$ when $k=2$.
		\end{itemize}
\begin{center} 
\begin{figure}
\setlength{\unitlength}{1.2cm}
\begin{picture}(5,6.3)(1.8,-4)
\put(0,0){
\begin{picture}(2,2)
\put(-1, 0){\line(1,2){1}} 
\put(0, 2){\line(1,-2){1}}
\put(-1,0){\line(1,0){2}}
\put(-1.,0){\circle*{0.1}}
\put(1.,0){\circle*{0.1}}
\put(0,2){\circle*{0.1}}
\put(0,0){\circle*{0.1}}
\put(0.5,1){\circle*{0.1}}
\put(-0.5,1){\circle*{0.1}}
\end{picture}
}

\put(1.5, 1){\vector(1, 0){1}}
\put(1.68, 1.15){$\nabla$}

\put(4,0){
\put(-1, 0){\line(1,2){1}} 
\put(0, 2){\line(1,-2){1}}
\put(-1,0){\line(1,0){2}}
\put(-0.16, 2.06){$\curl$}
\put(-1.14, -0.18){$\curl$}
\put(0.88, -0.18){$\curl$}
\put(-0.75, 0.65){\vector(1, 2){0.3}}
\put(-0.9, 0.35){\vector(1, 2){0.3}}
\put(0.75, 0.65){\vector(-1, 2){0.3}}
\put(0.9, 0.35){\vector(-1, 2){0.3}}
\put(-0.2, -0.05){\vector(1, 0){0.6}}
\put(-0.45, -0.05){\vector(1, 0){0.6}}
}

\put(5.5, 1){\vector(1, 0){1}}
\put(5.75, 1.1){{$\nabla\times$}}
\put(8,0){
\begin{picture}(2,2)
\put(-1, 0){\line(1,2){1}} 
\put(0, 2){\line(1,-2){1}}
\put(-1,0){\line(1,0){2}}
\put(-1, 0){\circle*{0.1}}
\put(1, 0){\circle*{0.1}}
\put(-0, 2){\circle*{0.1}}
\put(-0, 0.5){\circle*{0.1}}
\end{picture}
}

\put(0,-3){
\begin{picture}(2,2)
\put(-1, 0){\line(0,2){2}} 
\put(-1, 2){\line(1,0){2}}
\put(1,2){\line(0,-2){2}}
\put(-1,0){\line(2,0){2}}
\put(-1.,2){\circle*{0.1}}
\put(-1.,0){\circle*{0.1}}
\put(1,2){\circle*{0.1}}
\put(1,0){\circle*{0.1}}
\put(0,1){\circle*{0.1}}
\put(-1.,1){\circle*{0.1}}
\put(1,1){\circle*{0.1}}
\put(0,0){\circle*{0.1}}
\put(0,2){\circle*{0.1}}
\end{picture}
}

\put(1.5, -2){\vector(1, 0){1}}
\put(1.68, -1.8){$\nabla$}

\put(4,-3){
\put(-1, 0){\line(0,2){2}} 
\put(-1, 2){\line(1,0){2}}
\put(1,2){\line(0,-2){2}}
\put(-1,0){\line(2,0){2}}
\put(-1.2,2.06){$\curl$}
\put(-1.2,-0.2){$\curl$}
\put(0.8,2.06){$\curl$}
\put(0.8,-0.2){$\curl$}
\put(-0.1,0.9){$\times$}
\put(-0.2, 2.06){\vector(1, 0){0.6}}
\put(1.06, 0.65){\vector(0, 1){0.6}}
\put(-1.06, 0.65){\vector(0, 1){0.6}}
\put(-0.2, -0.05){\vector(1, 0){0.6}}
\put(-0.5, 2.06){\vector(1, 0){0.6}}
\put(1.06, 0.25){\vector(0, 1){0.6}}
\put(-1.06, 0.25){\vector(0, 1){0.6}}
\put(-0.5, -0.05){\vector(1, 0){0.6}}
}

\put(5.5, -2){\vector(1, 0){1}}
\put(5.75, -1.8){{$\nabla\times$}}
\put(8,-3){
\begin{picture}(2,2)
\put(-1, 0){\line(0,2){2}} 
\put(-1, 2){\line(1,0){2}}
\put(1,2){\line(0,-2){2}}
\put(-1,0){\line(2,0){2}}
\put(-1.,2){\circle*{0.1}}
\put(-1.,0){\circle*{0.1}}
\put(1,2){\circle*{0.1}}
\put(1,0){\circle*{0.1}}
\put(-0, 1){\circle*{0.1}}
\end{picture}}

\put(0,-4){$\Sigma_h^2$}
\put(1.5, -4){\vector(1, 0){1}}
\put(1.68, -3.8){$\nabla$}
\put(4,-4){$V_h^{1,2}$}
\put(5.5, -4){\vector(1, 0){1}}
\put(5.75, -3.8){{$\nabla\times$}}
\put(8,-4){$W_h^1$}
\end{picture}

\caption{The lowest-order ($k=2$) finite element complex \eqref{discrete-complex-k} in 2D.}\label{fig:firstfamily}
\end{figure}
\end{center}
\subsubsection{Rectangular elements}
Similarly, we can extend the rectangular elements to the case of $k=2$.
The DOFs for $\bm{u}\in V^{k-1, k}_h(K)=\nabla Q_k(K)\oplus \mathfrak p W^{k-1}_h(K)$ or $\nabla Q_k(K)\oplus \widetilde{\mathfrak p}W^{k-1}_h(K)$ are given by the following.
\begin{itemize}
		\item Vertex DOFs $\bm M_{ {v}}({\bm u})$ at all the vertices $ {v}_{i}$ of $K$:
	\begin{equation*}
	\bm M_{ {v}}({\bm u})=\left\{( \nabla\times {\bm u})(  v_{i}),\; i=1,\;2,\cdots,4\right\}.
	\end{equation*}
	\item Edge DOFs $\bm M_{ {e}}(  {\bm u})$ at all the edges $ {e}_i$ of $ {K}$, each with the unit  tangential vector $ {\bm \tau}_i$:
	\begin{align*}
		 \bm M_{ {e}}(  {\bm u})=&\left\{\int_{e_i} {\bm u}\cdot  {\bm \tau}_i  {q}\mathrm d {s},\ \forall  {q}\in P_{k-1}( {e}_i), i=1,2,\cdots,4\right\}\\
		 \cup&	\left\{\int_{e_i}\nabla\times{\bm u}q\d s,\ \forall  {q}\in P_{k-3}( {e}_i), i=1,2,\cdots,4\right\}.
	\end{align*}	
    \item Interior DOFs $\bm M_{ {K}}(  {\bm u})$:
	\begin{align*}
	&\bm M_{ {K}}(  {\bm u})=\left\{\int_{ {K}}  {\bm u}\cdot  {\bm q}\mathrm \d A,\ \forall \bm q \in  \mathcal{G}_1\oplus \mathcal{G}_2 \right\},
	\end{align*}
	where $\mathcal{G}_1=\big\{ {\bm q}\ |\  {\bm q}=  \psi {\bm x},\ \forall   \psi\in Q_{k-2}( K)\big\}\ \text{and}\ \mathcal{G}_2=\big\{ {\bm q}\ |\   {\bm q}= \bm\nabla\times {\varphi},\ \forall  {\varphi}\in  {Q}_{k-3}( K)\slash{\mathbb{R}}\big\}$
	when $k\geq 3$; $\mathcal{G}_1=\{{\bm x}\}$ and $\mathcal{G}_2=\emptyset$ when $k=2$.
	\end{itemize}

\begin{theorem}
If $\bm u\in \bm H^s(\Omega)$ and $\nabla\times\bm u\in  H^s(\Omega)$, $1+\delta\leq s\leq k$ with $\delta>0$, then we have the following error estimates for the interpolation $\Pi_h$,
	\begin{align}
	&\left\|\bm u-\Pi_h\bm u\right\|\leq Ch^{s}(\left\|\bm u\right\|_{s}+\left\|\nabla\times\bm u\right\|_{s}),\\
	&\left\|\nabla\times(\bm u-\Pi_h\bm u)\right\|\leq Ch^s\left\|\nabla\times\bm u\right\|_{s},\\
	&	\left\|(\nabla\times)^2(\bm u-\Pi_h\bm u)\right\|\leq Ch^{s-1}\left\|\nabla\times\bm u\right\|_{s}.
	\end{align}
\end{theorem}
\begin{proof}
From Lemma \ref{Vh}, $\bm P_{k-1}(K)\subseteq V^{k-1, k}_h(K)$ and  $\bm P_{k-1}(K)\subseteq W^{k-1}_h(K)$. \end{proof}

\subsection{A family of the curl-curl conforming elements with $r=k+1$}

We take $r=k+1$ in \eqref{local-complex} for $k\geq 2$ to get the following complex:
\begin{equation}\label{discrete-complex-k1}
\begin{tikzcd}
0 \arrow{r} &\mathbb{R} \arrow{r}{\subset} & \Sigma_h^{k+1} \arrow{r}{\nabla} &V_{h}^{k, k}   \arrow{r}{\nabla\times} & W_h^{k-1} \arrow{r}&0.
 \end{tikzcd}
\end{equation}

 We note that $V^{k, k}_h(K)=\bm P_k(K)$ when $k\geq 4$ and $K$ is a triangle, and thus $V_h^{k, k}(K)$ on triangles coincides with the finite elements constructed in \cite{quad-curl-eig-posterior} for $k\geq 4$. The lower-order triangular elements and the entire family of rectangular elements fill the gap in \cite{quad-curl-eig-posterior}.

The lowest-order cases are shown in Fig. \ref{fig:secondfamily}. The number of DOFs of the lowest-order element is 13 for a triangle and 20 for a rectangle.

\begin{center} 
\begin{figure}
\setlength{\unitlength}{1.2cm}
\begin{picture}(5,6.3)(1.8,-4)
\put(0,0){
\begin{picture}(2,2)
\put(-1, 0){\line(1,2){1}} 
\put(0, 2){\line(1,-2){1}}
\put(-1,0){\line(1,0){2}}
\put(-1.,0){\circle*{0.1}}
\put(1.,0){\circle*{0.1}}
\put(0,2){\circle*{0.1}}
\put(-0.3,0){\circle*{0.1}}
\put(0.3,0){\circle*{0.1}}
\put(-0.3,1.36){\circle*{0.1}}
\put(-0.66,0.67){\circle*{0.1}}
\put(0.3,1.36){\circle*{0.1}}
\put(0.66,0.67){\circle*{0.1}}
\put(0,0.67){\circle*{0.1}}
\end{picture}
}

\put(1.5, 1){\vector(1, 0){1}}
\put(1.68, 1.15){$\nabla$}

\put(4,0){
\put(-1, 0){\line(1,2){1}} 
\put(0, 2){\line(1,-2){1}}
\put(-1,0){\line(1,0){2}}
\put(-0.16, 2.06){$\curl$}
\put(-1.14, -0.18){$\curl$}
\put(0.88, -0.18){$\curl$}
\put(-0.75, 0.65){\vector(1, 2){0.3}}
\put(-0.9, 0.35){\vector(1, 2){0.3}}
\put(-0.62, 0.9){\vector(1, 2){0.3}}
\put(0.75, 0.65){\vector(-1, 2){0.3}}
\put(0.9, 0.35){\vector(-1, 2){0.3}}
\put(0.62, 0.9){\vector(-1, 2){0.3}}
\put(-0.2, -0.05){\vector(1, 0){0.6}}
\put(-0.45, -0.05){\vector(1, 0){0.6}}
\put(-0.62, -0.05){\vector(1, 0){0.5}}
\put(-0.13,0.66){$\times$}
}

\put(5.5, 1){\vector(1, 0){1}}
\put(5.75, 1.1){{$\nabla\times$}}
\put(8,0){
\begin{picture}(2,2)
\put(-1, 0){\line(1,2){1}} 
\put(0, 2){\line(1,-2){1}}
\put(-1,0){\line(1,0){2}}
\put(-1, 0){\circle*{0.1}}
\put(1, 0){\circle*{0.1}}
\put(-0, 2){\circle*{0.1}}
\put(-0, 0.5){\circle*{0.1}}
\end{picture}
}

\put(0,-3){
\begin{picture}(2,2)
\put(-1, 0){\line(0,2){2}} 
\put(-1, 2){\line(1,0){2}}
\put(1,2){\line(0,-2){2}}
\put(-1,0){\line(2,0){2}}
\put(-1.,2){\circle*{0.1}}
\put(-1.,0){\circle*{0.1}}
\put(1,2){\circle*{0.1}}
\put(1,0){\circle*{0.1}}
\put(-1.,1.35){\circle*{0.1}}
\put(-1.,0.66){\circle*{0.1}}
\put(-0.35,0.66){\circle*{0.1}}
\put(0.35,0.66){\circle*{0.1}}
\put(-0.35,1.35){\circle*{0.1}}
\put(0.35,1.35){\circle*{0.1}}
\put(1,1.35){\circle*{0.1}}
\put(1,0.66){\circle*{0.1}}
\put(-0.35,0){\circle*{0.1}}
\put(0.35,0){\circle*{0.1}}
\put(0.35,2){\circle*{0.1}}
\put(-0.35,2){\circle*{0.1}}
\end{picture}
}

\put(1.5, -2){\vector(1, 0){1}}
\put(1.68, -1.8){$\nabla$}

\put(4,-3){
\put(-1, 0){\line(0,2){2}} 
\put(-1, 2){\line(1,0){2}}
\put(1,2){\line(0,-2){2}}
\put(-1,0){\line(2,0){2}}
\put(-1.2,2.06){$\curl$}
\put(-1.2,-0.2){$\curl$}
\put(0.8,2.06){$\curl$}
\put(0.8,-0.2){$\curl$}
\put(0.23,0.66){$\times$}
\put(-0.4,0.66){$\times$}
\put(0.23,1.35){$\times$}
\put(-0.4,1.35){$\times$}
\put(-0.2, 2.06){\vector(1, 0){0.6}}
\put(1.06, 0.75){\vector(0, 1){0.6}}
\put(-1.06, 0.75){\vector(0, 1){0.6}}
\put(-0.2, -0.05){\vector(1, 0){0.6}}
\put(-0.5, 2.06){\vector(1, 0){0.6}}
\put(1.06, 0.35){\vector(0, 1){0.6}}
\put(-1.06, 0.35){\vector(0, 1){0.6}}
\put(-0.5, -0.05){\vector(1, 0){0.6}}
\put(-0.8, 2.06){\vector(1, 0){0.6}}
\put(1.06, -0){\vector(0, 1){0.6}}
\put(-1.06, -0){\vector(0, 1){0.6}}
\put(-0.5, -0.05){\vector(1, 0){0.6}}
\put(-0.8, -0.05){\vector(1, 0){0.6}}
}

\put(5.5, -2){\vector(1, 0){1}}
\put(5.75, -1.8){{$\nabla\times$}}
\put(8,-3){
\begin{picture}(2,2)
\put(-1, 0){\line(0,2){2}} 
\put(-1, 2){\line(1,0){2}}
\put(1,2){\line(0,-2){2}}
\put(-1,0){\line(2,0){2}}
\put(-1.,2){\circle*{0.1}}
\put(-1.,0){\circle*{0.1}}
\put(1,2){\circle*{0.1}}
\put(1,0){\circle*{0.1}}
\put(-0, 1){\circle*{0.1}}
\end{picture}
}
\put(0,-4){$\Sigma_h^3$}
\put(1.5, -4){\vector(1, 0){1}}
\put(1.68, -3.8){$\nabla$}
\put(4,-4){$V_h^{2,2}$}
\put(5.5, -4){\vector(1, 0){1}}
\put(5.75, -3.8){{$\nabla\times$}}
\put(8,-4){$W_h^1$}
\end{picture}

\caption{The lowest-order ($k=2$) finite element complex \eqref{discrete-complex-k1} in 2D. }
\label{fig:secondfamily}
\end{figure}
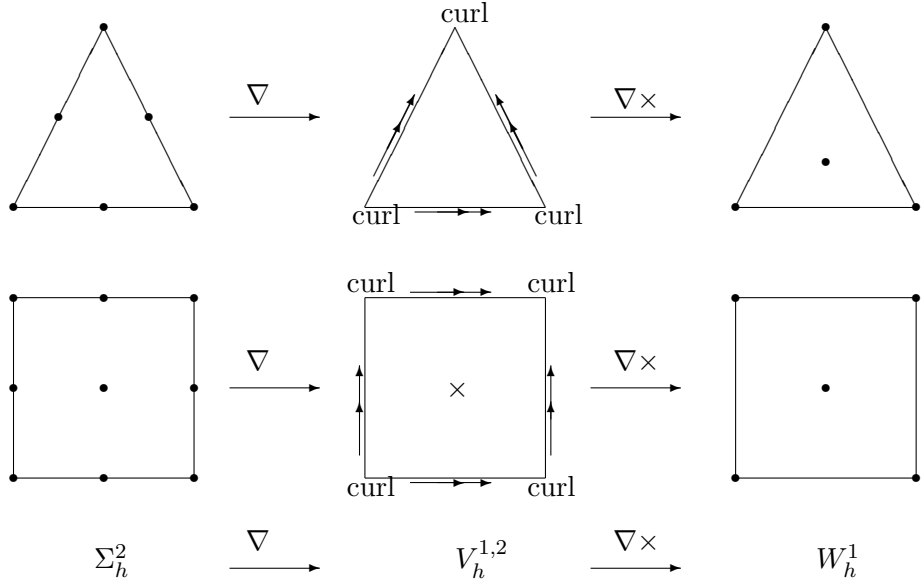
\end{center}

\subsubsection{Triangular elements}
The DOFs for $\bm{u}\in V^{k, k}_h(K)=\nabla P_{k+1}(K)\oplus \mathfrak p W^{k-1}_h(K)$ are given as follows.

\begin{itemize}
		\item Vertex DOFs $\bm M_{ {v}}({\bm u})$ at all the vertices $ {v}_{i}$ of $K$:
	\begin{equation*}
	\bm M_{ {v}}({\bm u})=\left\{( \nabla\times {\bm u})(  v_{i}),\; i=1,\;2,\;3\right\}.
	\end{equation*}
	\item Edge DOFs $\bm M_{ {e}}(  {\bm u})$ at all the edges $ {e}_i$ of $ {K}$, each with the unit  tangential vector $ {\bm \tau}_i$:
	\begin{align*}
		 \bm M_{ {e}}(  {\bm u})=&\left\{\int_{e_i} {\bm u}\cdot  {\bm \tau}_i  {q}\mathrm d {s},\ \forall  {q}\in P_{k}( {e}_i), i=1,2,3\right\}\\
		 \cup&	\left\{\int_{e_i}\nabla\times{\bm u}q\d s,\ \forall  {q}\in P_{k-3}( {e}_i), i=1,2,3\right\}.
	\end{align*}	
    \item Interior DOFs $\bm M_{ {K}}(  {\bm u})$:
	\begin{align*}
	&\bm M_{ {K}}(  {\bm u})=\left\{\int_{ {K}}  {\bm u}\cdot  {\bm q}\mathrm \d A,\ \forall \bm q \in  \mathcal{D} \right\},
	\end{align*}
	where $\mathcal{D}=\bm P_{k-5}(  K)\oplus\widetilde{P}_{k-5} {\bm x}\oplus\widetilde{P}_{k-4} {\bm x}\oplus
	\widetilde{P}_{k-3} {\bm x}\oplus
	\widetilde{P}_{k-2} {\bm x}$ when $k\geq 5$; $\mathcal{D}={P}_{k-2} {\bm x}$ when $k=2,3,4$.
		\end{itemize}

\subsubsection{Rectangular elements} We extend the construction in \cite{quad-curl-eig-posterior} to the rectangular case.
The DOFs for $\bm{u}\in V^{k, k}_{h}(K)=\nabla Q_{k+1}(K)\oplus \mathfrak p W^{k-1}_h(K)$ are given by the following.
\begin{itemize}
		\item Vertex DOFs $\bm M_{ {v}}({\bm u})$ at all the vertices $ {v}_{i}$ of $K$:
	\begin{equation*}
	\bm M_{ {v}}({\bm u})=\left\{( \nabla\times {\bm u})(  v_{i}),\; i=1,\;2,\cdots,4\right\}.
	\end{equation*}
	\item Edge DOFs $\bm M_{ {e}}(  {\bm u})$ at all the edges $ {e}_i$ of $ {K}$, each with the unit  tangential vector $ {\bm \tau}_i$:
	\begin{align*}
		 \bm M_{ {e}}(  {\bm u})=&\left\{\int_{e_i} {\bm u}\cdot  {\bm \tau}_i  {q}\mathrm d {s},\ \forall  {q}\in P_{k}( {e}_i), i=1,2,\cdots,4\right\}\\
		 \cup&	\left\{\int_{e_i}\nabla\times{\bm u}q\d s,\ \forall  {q}\in P_{k-3}( {e}_i), i=1,2,\cdots,4\right\}.
	\end{align*}	
    \item Interior DOFs $\bm M_{ {K}}(  {\bm u})$:
	\begin{align*}
	&\bm M_{ {K}}(  {\bm u})=\left\{\int_{ {K}}  {\bm u}\cdot  {\bm q}\mathrm \d A,\ \forall \bm q \in  \mathcal{G}_1\oplus \mathcal{G}_2 \right\},
	\end{align*}
	where $\mathcal{G}_1=\big\{ {\bm q}\ |\  {\bm q}=  \psi{\bm x},\ \forall   \psi\in Q_{k-1}(  K)\big\}\ \text{and}\ \mathcal{G}_2=\big\{ {\bm q}\ |\   {\bm q}= \bm{\nabla}\times {\varphi},\ \forall  {\varphi}\in  {Q}_{k-3}( K)\slash{\mathbb{R}}\big\}$
	when $k\geq 3$; $\mathcal{G}_2=\emptyset$ when $k=2$.
	\end{itemize}
This family of elements lead to one-order higher accuracy in $L^2$-norms.
\begin{theorem}\label{interp-f2}
If $\bm u\in \bm H^{s+1}(\Omega)$, $ 1+\delta\leq s \leq k$ with $\delta>0$, then we have the following error estimates for the interpolation $\Pi_h$,
	\begin{align}
	&\left\|\bm u-\Pi_h\bm u\right\|\leq Ch^{s+1}\left\|\bm u\right\|_{s+1},\\
	&\left\|\nabla\times(\bm u-\Pi_h\bm u)\right\|\leq Ch^s\left\|\bm u\right\|_{s+1},\\
	&	\left\|(\nabla\times)^2(\bm u-\Pi_h\bm u)\right\|\leq Ch^{s-1}\left\|\bm u\right\|_{s+1}.
	\end{align}
\end{theorem}
\begin{proof}
From Lemma \ref{Vh}, $\bm P_{k}(K)\subseteq V^{k, k}_h(K)$ and  $\bm P_{k-1}(K)\subseteq W^{k-1}_h(K)$. 
\end{proof}
 
\begin{remark}
	By the duality argument, in the sense of the $L^2$-norm, the numerical solution $\bm u_h$ converges to the exact solution $\bm u$ with an order $\min\{s+1,2(s-1)\}$.  Hence when $s<3$ the convergence order is $2(s-1)$.
	\end{remark}
\section{Numerical Experiments}
In this section, we use the three families of the $H(\text{curl}^2)$-conforming finite elements  to solve the quad-curl problem: \\For $\bm  f\in H(\div^0;\Omega)$, find $\bm u$, such that
\begin{equation}\label{prob1}
\begin{split}
(\nabla\times)^4\bm u+\bm u&=\bm f\ \ \text{in}\;\Omega,\\
\nabla \cdot \bm u &= 0\ \ \text{in}\;\Omega,\\
\bm u\times\bm n&=0\ \ \text{on}\;\partial \Omega,\\
\nabla \times \bm u&=0\ \  \text{on}\;\partial \Omega.
\end{split}
\end{equation}
Here $H(\div^0;\Omega)$ is the space of $\bm L^2(\Omega)$ functions with vanishing divergence, i.e., 
\[H(\text{div}^0;\Omega) :=\{\bm u\in {\bm L}^2(\Omega):\; \nabla\cdot \bm u=0\},\] and  $\bm n$ is the unit outward normal vector to $\partial \Omega$.
Taking divergence on both sides of the first equation of \eqref{prob1}, we see that the divergence-free condition $\nabla\cdot\bm u=0$  holds automatically.

We define $H_0(\text{curl}^2;\Omega)$ with vanishing boundary conditions:
\begin{align*}
&H_0(\text{curl}^2;\Omega):=\{\bm u \in H(\text{curl}^2;\Omega):\;{\bm n}\times\bm u=0\; \text{and}\; \nabla\times \bm u=0\;\; \text{on}\ \partial \Omega\}.
\end{align*}
The variational formulation reads:
find $\bm u\in H_0(\curl^2;\Omega)$,  such that
\begin{equation}\label{prob22}
\begin{split}
a(\bm u,\bm v)&=(\bm f, \bm v)\quad \forall \bm v\in H_0(\curl^2;\Omega),
\end{split}
\end{equation}
with $a(\bm u,\bm v):=(\nabla\times\nabla\times\bm u,\nabla\times\nabla\times\bm v) + (\bm u,\bm v)$.

 We define the finite element space with vanishing boundary conditions
\begin{eqnarray*}
  V^0_h=\{\bm{v}_h\in V^{r-1,k}_h,\ \bm{n} \times \bm{v}_h=0\ \text{and}\ \nabla\times  \bm{v}_h = 0 \ \text {on} \ \partial\Omega\}.
\end{eqnarray*}

\begin{remark}
To enforce the vanishing boundary conditions, we only need to set all the DOFs on $\partial\Omega$ to be 0.
\end{remark}

The $H(\text{curl}^2)$-conforming finite element method reads: seek $\bm u_h\in V^0_h$,  such that
\begin{equation}\label{prob3}
\begin{split}
 a(\bm u_h,\bm v_h)&=(\bm f, \bm v_h)\quad \forall \bm v_h\in V^0_h.
\end{split}
\end{equation}

We now turn to a concrete example.
We consider the problem \eqref{prob1} on a unit square $\Omega=(0,1)\times(0,1)$ with an exact solution
	\begin{equation}
	\bm u=\left(
	\begin{array}{c}
	3\pi\sin^3(\pi x)\sin^2(\pi y)\cos(\pi y) \\
	-3\pi \sin^3(\pi y)\sin^2(\pi x)\cos(\pi x)\\
	\end{array}
	\right).
	\end{equation}
   Then the source term $\bm f$ can be obtained by a simple calculation.

   The finite element solution is denoted as $\bm u_h$. To measure the error between the exact solution and the finite element solution, we denote 
	\[\bm e_h=\bm u-\bm u_h.\]

\subsection{The new family of elements with $r=k-1$}
We  first use the lowest-order element in the new family with $r=k-1$ to solve the problem \eqref{prob1}. In this test, we use  uniform triangular meshes and uniform rectangular meshes with the mesh size $h$ varying  from ${1}/{20}$ to ${1}/{320}$ with the bisection strategy. For $\bm u=(u_1,u_2)^T$, we define two discrete norms:
\begin{align}
\3bar\bm u\3bar_{V}^2=& 	\sum_{K\in \mathcal{T}_h}2h_x^K\int_{y_c^K-h_y^K}^{y_c^K+h_y^K}u_1^2(x_c^K,y)\d y+\sum_{K\in \mathcal{T}_h}2h_y^K\int_{x_c^K-h_x^K}^{x_c^K+h_x^K}u_2^2(x,y_c^K)\d x,\label{u_discrete_norm}\\
&\3bar \bm u\3bar^2_W=\sum_{K\in \mathcal{T}_h}4h_x^Kh_y^K\left[u_1^2(x_c^K,y_c^K)+u_2^2(x_c^K,y_c^K)\right]\label{curlcurlu_discrete_norm}
\end{align}
where $K=(x_c^K-h_x^K,x_c^K+h_x^K)\times(y_c^K-h_y^K,y_c^K+h_y^K)$ and $x_c^K,y_c^K,h_x^K,h_y^K$ are defined in Fig. \ref{rect}.
\begin{figure}[!h]
\centering
\includegraphics[width=7cm]{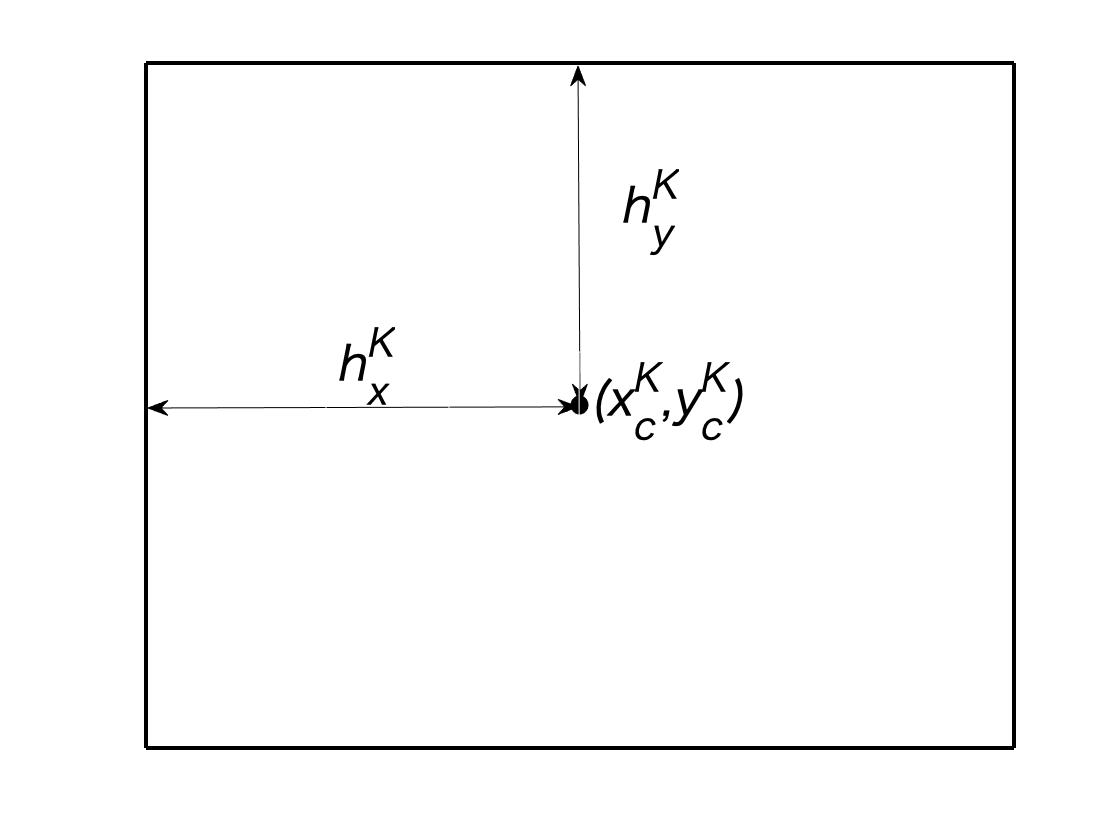}
\caption{A rectangular element}
\label{rect}
\end{figure}

Table \ref{tab1} illustrates various errors and convergence rates for triangular elements.  
Table \ref{tab2} shows errors measured in various norms for rectangular elements. We also depict error curves for rectangular elements with a log-log scale in Fig. \ref{fig1}.
We observe that the numerical solution converges to the exact solution with a convergence order 1 in the $L^2$-norm, 2 in the $H(\curl)$-norm, and 1 in the $H(\curl^2)$-norm, respectively. From Fig. \ref{fig1}, we also observe some superconvergence phenomena of $\bm e_h$ and $(\nabla\times)^2\bm e_h$ measured in the sense of  \eqref{u_discrete_norm} and \eqref{curlcurlu_discrete_norm}, respectively. Using these superconvergent results, together with some recovery techniques, we can construct a solution with higher accuracy if needed.

\begin{table}[h!]
	\centering
	\caption{Numerical results by  the lowest-order $(k=2)$ triangular element in the new family $(r=k-1)$ of $H(\text{curl}^2)$-conforming elements}\label{tab1}
	\begin{tabular}{cccccccc}
		\hline
		$h$ &$\left\|\bm e_h\right\|$&rates&$\left\|\nabla\times\bm e_h\right\|$&rates&$\left\|(\nabla\times)^2\bm e_h\right\|$& rates\\
		\hline
		$1\slash 20$&1.90386e-02  &    &4.92128e-02        &&2.49140e+00\\
		$1\slash40$ &9.46304e-03&1.0086 &1.25357e-02& 1.9730 &1.25626e+00&0.9878 \\
		$1\slash80$ &4.72423e-03 &1.0022&3.14876e-03&1.9932  &6.29464e-01& 0.9969 \\
		$1\slash160$&2.36120e-03 &1.0006&7.88122e-04 &1.9983 &3.14900e-01&  0.9992  \\
		$1\slash320$&1.18329e-03&0.9967 &1.97108e-04&1.9994  &1.57471e-01&0.9998\\
		\hline
	\end{tabular}
\end{table}


 \begin{table}[h!]
	\centering
	\caption{Numerical results by the lowest-order $(k=2)$ rectangular element in the new family $(r=k-1)$ of $H(\text{curl}^2)$-conforming elements}\label{tab2}
	\begin{tabular}{cccccccc}
		\hline
		$h$ &$\left\|\bm e_h\right\|$&$\left\|\bm e_h\right\|_V$&$\left\|\nabla\times\bm e_h\right\|$&$\left\|(\nabla\times)^2\bm e_h\right\|$& $\left\|(\nabla\times)^2\bm e_h\right\|_W$\\
		\hline
  $1\slash 20$  & 1.1286e-01  & 1.4312e-02  & 1.3911e-01 &  1.2610e+01  &2.0177e+00\\
   $1\slash 40$ & 5.6602e-02  & 3.5786e-03  & 3.4624e-02 &  6.2788e+00  &5.0321e-01\\
   $1\slash 80$ & 2.8323e-02  & 8.9473e-04  & 8.6464e-03 &  3.1361e+00  &1.2573e-01\\
  $1\slash 160$ & 1.4164e-02  & 2.2375e-04  & 2.1610e-03 &  1.5676e+00  &3.1428e-02\\
  $1\slash 320$ & 7.0832e-03  & 1.1206e-04  & 5.4022e-04 &  7.8375e-01  &7.8567e-03\\
		\hline
	\end{tabular}
\end{table}

\begin{figure}[!h]
\centering
\includegraphics[width=12cm]{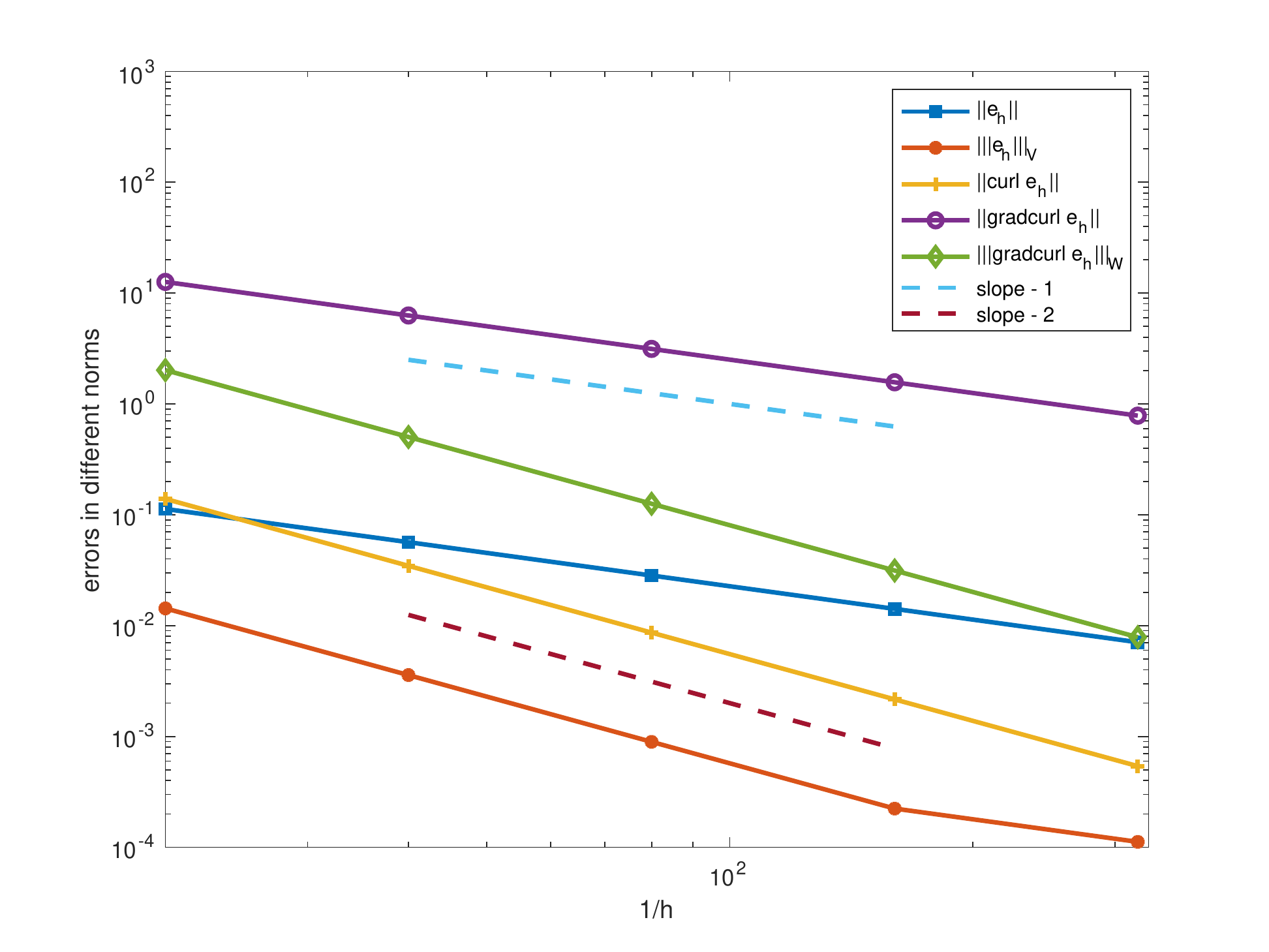}
\caption{Error curves in different norms}\label{fig1}
\end{figure}
 
\subsection{The family of elements with $r=k$}
We then use the lowest-order element $V^{k-1, k}_{h}$ in the  family with $r=k$. 

Again, we use the uniform mesh. Table \ref{tab3} and Table \ref{tab4} demonstrate the numerical results with  $h$ varying from ${1}/{10}$ to ${1}/{160}$. We observe a second-order convergence in the $L^2$-norm and $H(\curl)$-norm, and a first-order convergence in the $H(\curl^2)$-norm.
\begin{table}[h!]
	\centering
	\caption{Numerical results by the lowest-order $(k=2)$  triangular element in the family  of $H(\text{curl}^2)$-conforming elements with $ r=k$} \label{tab3}
	\begin{tabular}{cccccccc}
		\hline
		$h$ &$\left\|\bm e_h\right\|$&rates&$\left\|\nabla\times\bm e_h\right\|$&rates&$\left\|(\nabla\times)^2\bm e_h\right\|$& rates\\
		\hline
	  $1\slash 10$ &1.946294e-02&           &1.831378e-01 &         &4.821773e+00\\
    $1\slash 20$   &5.104203e-03  & 1.9310  &4.921121e-02&1.8959    &2.491403e+00&0.9526\\
     $1\slash 40$  &1.292287e-03 & 1.9818   &1.253529e-02 & 1.9730  &1.256258e+00&0.9878\\
     $1\slash 80$  &3.241096e-04 &  1.9954  &3.148659e-03  &1.9932  &6.294644e-01&0.9969\\
     $1\slash 160$ &8.131642e-05&   1.9949  &7.880957e-04  & 1.9983 &3.148996e-01&0.9992	\\	\hline
	\end{tabular}
\end{table}

\begin{table}[h!]
	\centering
	\caption{Numerical results by the lowest-order $(k=2)$  rectangular element in the  family  of $H(\text{curl}^2)$-conforming elements with $ r=k$} \label{tab4}
	\begin{tabular}{cccccccc}
		\hline
		$h$ &$\left\|\bm e_h\right\|$&rates&$\left\|\nabla\times\bm e_h\right\|$&rates&$\left\|(\nabla\times)^2\bm e_h\right\|$& rates\\
		\hline
    $1\slash 10$&6.449132e-02&          &5.664956e-01&          &2.563424e+01&\\
    $1\slash 20$&1.592685e-02&2.0176    &1.391017e-01&2.0259    &1.261045e+01&1.0235\\
    $1\slash 40$&3.970283e-03&2.0041    &3.462207e-02&2.0064    &6.278774e+00&1.0061\\
    $1\slash 80$&9.918685e-04&2.0010    &8.645999e-03&2.0016    &3.136060e+00&1.0015\\
   $1\slash 160$&2.480152e-04&1.9997    &2.160906e-03&2.0004    &1.567613e+00&1.0004	\\	\hline
	\end{tabular}
\end{table}

\subsection{The family of elements with $r=k+1$}
We now test elements in the family with $r=k+1$. We apply  the same mesh as before. Table \ref{tab5}, Table \ref{tab6}, and Table \ref{tab7} show the  numerical results for various mesh sizes and elements.  We observe  the same convergence behavior as in Theorem \ref{interp-f2}. 

\begin{table}[h]
	\centering
	\caption{Numerical results by the lowest-order $(k=2)$  triangular element in the  family of $H(\text{curl}^2)$-conforming elements with $r=k+1$} \label{tab5}
	\begin{tabular}{cccccccc}
		\hline
		$h$ &$\left\|\bm e_h\right\|$&rates&$\left\|\nabla\times\bm e_h\right\|$&rates&$\left\|(\nabla\times)^2\bm e_h\right\|$& rates\\
		\hline
		$1\slash 10$ & 1.916204e-01  && 1.831377e+00&&4.821773e+01\\
		$1\slash20$& 4.953536e-02&1.9517&4.921121e-01&1.8959&2.491403e+01 		&0.9526 \\
		$1\slash40$&1.254233e-02 &1.9817&1.253529e-01&1.9730&1.256258e+01& 0.9878 \\
		$1\slash80$&3.145763e-03& 1.9953& 3.148659e-02&1.9932&6.294644e+00 &0.9969 \\
		$1\slash160$&7.897003e-04& 1.9940&7.880958e-03&1.9983&3.148996e+00&0.9992\\
		\hline
	\end{tabular}
\end{table}

\begin{table}[!ht]
	\centering
	\caption{Numerical results by the lowest-order $(k=2)$  rectangular element in the family of $H(\text{curl}^2)$-conforming elements with $r=k+1$} \label{tab6}
	\begin{tabular}{cccccccc}
		\hline
		$h$ &$\left\|\bm e_h\right\|$&rates&$\left\|\nabla\times\bm e_h\right\|$&rates&$\left\|(\nabla\times)^2\bm e_h\right\|$& rates\\
		\hline
		$1\slash 10$ &8.399241e-02 &&  7.736407e-01& &3.117602e+01\\
		$1\slash20$& 2.055671e-02&2.0306 &1.924122e-01&2.0075&1.556987e+01&1.0017 \\
		$1\slash40$&5.125523e-03&2.0038&4.804486e-02&2.0017&7.783057e+00&1.0003 \\
		$1\slash80$&1.280556e-03& 2.0009&1.200764e-02&2.0004&3.891305e+00&1.0001\\
		$1\slash160$&3.203172e-04& 1.9992& 3.001689e-03&2.0001&1.945625e+00&1.0000 \\
		\hline
	\end{tabular}
\end{table}

    \begin{table}[!ht]
	\centering
	\caption{Numerical results by the third-order ($k=3$)  rectangular element in the  family of $H(\text{curl}^2)$-conforming elements with $r=k+1$} \label{tab7}
	\begin{tabular}{cccccccc}
		\hline
		$h$ &$\left\|\bm e_h\right\|$&rates&$\left\|\nabla\times\bm e_h\right\|$&rates&$\left\|(\nabla\times)^2\bm e_h\right\|$& rates\\
		\hline
		$1\slash 4$ &6.482470e-02 && 9.955505e-01&&2.796216e+01& \\
		$1\slash 8$&4.580398e-03&3.8230&1.388809e-01& 2.8416&7.337119e+00&1.9302   \\
		$1\slash 16$&2.927226e-04&3.9679&1.780427e-02&2.9636&1.854476e+00& 1.9842\\
		$1\slash 32$&1.838464e-05& 3.9930&2.239038e-03&2.9913&4.648552e-01&1.9962
\\
$1\slash 64$&1.166284e-06& 3.9785&2.802981e-04&2.9978&1.162907e-01&1.9990
\\

			\hline
	\end{tabular}
\end{table}

We conclude this section by pointing out that the three families of elements bear their own advantages. 
The new family ($r=k-1$) can be the best choice if we pursue a low computational cost, while the  family with $r=k+1$ stands out for its higher accuracy in the $L^2$-norm.


\section{Conclusion}
In this paper, we constructed finite element de Rham complexes with enhanced smoothness in 2D. The new construction yields several curl-curl conforming elements. The two existing families of elements fit into our complexes and with the idea in this paper, we extend these elements to lower-order cases. The simple elements (e.g., with 6 DOFs and 8 DOFs for the lowest-order cases on  triangles and rectangles, respectively) are thus easy to implement.

In the future, we will construct discrete Stokes type complexes and curl-curl conforming elements in 3D and further investigate the superconvergence phenomena.
  
\section*{Acknowledgments}
We would like to thank Dr. Baiju Zhang for finding a mistake in the previous version.

\section*{Appendix}
The basis functions for the lowest-order element $(r=1,k=2)$ on a reference rectangle $(-1,1)\times(-1,1)$ are listed as follows. 
\begin{align*}
u_1 =[&-((x_2^2 - 1)(- 3x_2x_1^2 + 2x_1 + 5x_2 - 4))/32,\\
             & ((x_2^2 - 1)(3x_2x_1^2 + 2x_1 - 5x_2 + 4))/32,\\
             & ((x_2^2 - 1)(3x_2x_1^2 + 2x_1 - 5x_2 - 4))/32,\\
           &-((x_2^2 - 1)(- 3x_2x_1^2 + 2x_1 + 5x_2 + 4))/32,\\
 &-((x_2 - 1)(3x_1^2x_2^2 + 3x_1^2x_2 - 5x_2^2 - 5x_2 + 8))/32,\\
 & ((x_2 + 1)(3x_1^2x_2^2 - 3x_1^2x_2 - 5x_2^2 + 5x_2 + 8))/32,\\
                         & (x_2(x_2^2 - 1)(3x_1^2 - 5))/32,\\
                         &-(x_2(x_2^2 - 1)(3x_1^2 - 5))/32];
 \end{align*}
 \begin{align*}
u_2 =[&((x_1^2 - 1)(- 3x_1x_2^2 + 2x_2 + 5x_1 - 4))/32,\\
           & ((x_1^2 - 1)(- 3x_1x_2^2 - 2x_2 + 5x_1 + 4))/32,\\
           &  -((x_1^2 - 1)(3x_1x_2^2 + 2x_2 - 5x_1 + 4))/32,\\
           & ((x_1^2 - 1)(- 3x_1x_2^2 + 2x_2 + 5x_1 + 4))/32,\\
                        &  (x_1(x_1^2 - 1)(3x_2^2 - 5))/32,\\
                       &  -(x_1(x_1^2 - 1)(3x_2^2 - 5))/32,\\
 &-((x_1 - 1)(3x_1^2x_2^2 - 5x_1^2 + 3x_1x_2^2 - 5x_1 + 8))/32,\\
 & ((x_1 + 1)(3x_1^2x_2^2 - 5x_1^2 - 3x_1x_2^2 + 5x_1 + 8))/32].
  \end{align*}
  
The basis functions for the lowest-order element $(r=1,k=2)$ on a reference triangle are listed as follows. 
\begin{align*}
	u_1 =[&x_2(x_1 + x_2)/2 - x_2/2 + x_1x_2/2 + x_2(3x_1 - 4x_1x_2)(x_1 + x_2 - 1),\\
                                 & -x_1x_2(4x_2 - 3)(x_1 + x_2 - 1),\\
                          & x_1x_2 + x_2(3x_1 - 4x_1x_2)(x_1 + x_2 - 1),\\
         &- x_2(x_1 + x_2) - 3x_1x_2 - 6x_2(3x_1 - 4x_1x_2)(x_1 + x_2 - 1),\\
         &- x_2(x_1 + x_2) - 3x_1x_2 - 6x_2(3x_1 - 4x_1x_2)(x_1 + x_2 - 1),\\
      & 1 - 3x_1x_2 - 6x_2(3x_1 - 4x_1x_2)(x_1 + x_2 - 1) - x_2(x_1 + x_2)];
 \end{align*}
 \begin{align*}
u_2 =[&x_1(8x_1^2x_2 + 2x_1^2 + 8x_1x_2^2 - 6x_1x_2 - 3x_1 + 1)/2,\\
                           &x_1^2(4x_2 + 1)(x_1 + x_2 - 1),\\
                     &x_1x_2 + x_1(x_1 + 4x_1x_2)(x_1 + x_2 - 1),\\
    & x_1(x_1 + x_2) - 3x_1x_2 - 6x_1(x_1 + 4x_1x_2)(x_1 + x_2 - 1),\\
 &x_1(x_1 + x_2) - 3x_1x_2 - 6x_1(x_1 + 4x_1x_2)(x_1 + x_2 - 1) - 1,\\
    & x_1(x_1 + x_2) - 3x_1x_2 - 6x_1(x_1 + 4x_1x_2)(x_1 + x_2 - 1)].
     \end{align*}

\bibliographystyle{plain}
\bibliography{quadcurl-2d-reduction}{}

\end{document}